\documentclass[12pt,english,fleqn,liststotoc,bibtotoc,idxtotoc,BCOR7.5mm,tablecaptionabove]{amsart}
\usepackage[T1]{fontenc}
\usepackage[latin1]{inputenc}
\usepackage{geometry}
\usepackage{color}
\usepackage{amsthm}
\usepackage{amstext}
\usepackage{amssymb}
\usepackage{amsmath}
\usepackage{graphicx}
\usepackage{young}
\usepackage{xypic}
\usepackage{amscd}
\usepackage[section]{placeins}
\usepackage[svgnames]{xcolor}
\usepackage[unicode=true,
bookmarks=true,bookmarksnumbered=true,bookmarksopen=false,
breaklinks=false,pdfborder={0 0 1},backref=false,colorlinks=true]
{hyperref}
\hypersetup{pdftitle={},
	pdfauthor={Nick Early},
	pdfsubject={},
	pdfkeywords={},
	linkcolor=black, citecolor=black, urlcolor=blue, filecolor=blue,  pdfpagelayout=OneColumn, pdfnewwindow=true,  pdfstartview=XYZ, plainpages=false, pdfpagelabels}
\usepackage{breakurl}

\makeatletter

\theoremstyle{plain}
\newtheorem{thm}{\protect\theoremname}
\theoremstyle{plain}

\theoremstyle{plain}

\theoremstyle{remark}

\theoremstyle{plain}
\newtheorem{lem}[thm]{\protect\lemmaname}
\theoremstyle{plain}
\newtheorem{prop}[thm]{\protect\propositionname}
\theoremstyle{remark}

\theoremstyle{remark}
\newtheorem{rem}[thm]{\protect\remarkname}
\theoremstyle{definition}
\newtheorem{defn}[thm]{\protect\definitionname}
\theoremstyle{definition}
\newtheorem{example}[thm]{\protect\examplename}
\theoremstyle{plain}
\newtheorem{cor}[thm]{\protect\corollaryname}

\newcommand{\cellsize}{20}
\newlength{\cellsz} 
\newsavebox{\cell}
\sbox{\cell}{\begin{picture}(\cellsize,\cellsize)
	\put(0,0){\line(1,0){\cellsize}}
	\put(0,0){\line(0,1){\cellsize}}
	\put(\cellsize,0){\line(0,1){\cellsize}}
	\put(0,\cellsize){\line(1,0){\cellsize}}
	\end{picture}}
\newcommand\cellify[1]{\def\thearg{#1}\def\nothing{}%
	\ifx\thearg\nothing
	\vrule width0pt height\cellsz depth0pt\else
	\hbox to 0pt{\usebox{\cell} \hss}\fi%
	\vbox to \cellsz{
		\vss
		\hbox to \cellsz{\hss$#1$\hss}
		\vss}}
\newcommand\tableau[1]{\vtop{\let\\\cr
		\baselineskip -16000pt \lineskiplimit 16000pt \lineskip 0pt
		\ialign{&\cellify{##}\cr#1\crcr}}}
\newcommand{\kellsize}{20}
\newlength{\kellsz} 
\newsavebox{\kell}
\sbox{\kell}{\begin{picture}(\kellsize,\kellsize)
	\put(0,0){\line(1,0){\kellsize}}
	\put(0,0){\line(0,1){\kellsize}}
	\put(\kellsize,0){\line(0,1){\kellsize}}
	\put(0,\kellsize){\line(1,0){\kellsize}}
	\end{picture}}
\newcommand\kellify[1]{\def\thearg{#1}\def\nothing{}%
	\ifx\thearg\nothing
	\vrule width0pt height\kellsz depth0pt\else
	\hbox to 0pt{\usebox{\kell} \hss}\fi%
	\vbox to \kellsz{
		\vss
		\hbox to \kellsz{\hss$#1$\hss}
		\vss}}
\newcommand\ktableau[1]{\vtop{\let\\\cr
		\baselineskip -16000pt \lineskiplimit 16000pt \lineskip 0pt
		\ialign{&\kellify{##}\cr#1\crcr}}}

\newcommand{\sellsize}{10}
\newlength{\sellsz} 
\newsavebox{\sell}
\sbox{\sell}{\begin{picture}(\sellsize,8)
	\put(0,0){\line(1,0){\sellsize}}
	\put(0,0){\line(0,1){\sellsize}}
	\put(\sellsize,0){\line(0,1){\sellsize}}
	\put(0,\sellsize){\line(1,0){\sellsize}}
	\end{picture}}
\newcommand\sellify[1]{\def\thearg{#1}\def\nothing{}%
	\ifx\thearg\nothing
	\vrule width0pt height\sellsz depth0pt\else
	\hbox to 0pt{\usebox{\sell} \hss}\fi%
	\vbox to \sellsz{
		\vss
		\hbox to \sellsz{\hss$#1$\hss}
		\vss}}
\newcommand\stableau[1]{\vtop{\let\\\cr
		\baselineskip -16000pt \lineskiplimit 16000pt \lineskip 0pt
		\ialign{&\sellify{##}\cr#1\crcr}}}

\newcommand{\ssellsize}{5}
\newlength{\ssellsz} 
\newsavebox{\ssell}
\sbox{\ssell}{\begin{picture}(\ssellsize,4)
	\put(0,0){\line(1,0){\ssellsize}}
	\put(0,0){\line(0,1){\ssellsize}}
	\put(\ssellsize,0){\line(0,1){\ssellsize}}
	\put(0,\ssellsize){\line(1,0){\ssellsize}}
	\end{picture}}
\newcommand\ssellify[1]{\def\thearg{#1}\def\nothing{}%
	\ifx\thearg\nothing
	\vrule width0pt height\sellsz depth0pt\else
	\hbox to 0pt{\usebox{\ssell} \hss}\fi%
	\vbox to \ssellsz{
		\vss
		\hbox to \ssellsz{\hss$#1$\hss}
		\vss}}
\newcommand\sstableau[1]{\vtop{\let\\\cr
		\baselineskip -16000pt \lineskiplimit 16000pt \lineskip 0pt
		\ialign{&\ssellify{##}\cr#1\crcr}}}

\usepackage{ifpdf}
\ifpdf

\IfFileExists{lmodern.sty}
{\usepackage{lmodern}}{}

\fi 

\usepackage{multicol}

\makeatother

\providecommand{\claimname}{\inputencoding{latin9}Claim}
\providecommand{\conjecturename}{\inputencoding{latin9}Conjecture}
\providecommand{\corollaryname}{\inputencoding{latin9}Corollary}
\providecommand{\definitionname}{\inputencoding{latin9}Definition}
\providecommand{\examplename}{\inputencoding{latin9}Example}
\providecommand{\lemmaname}{\inputencoding{latin9}Lemma}
\providecommand{\notename}{\inputencoding{latin9}Note}
\providecommand{\propositionname}{\inputencoding{latin9}Proposition}
\providecommand{\questionname}{\inputencoding{latin9}Question}
\providecommand{\remarkname}{\inputencoding{latin9}Remark}
\providecommand{\theoremname}{\inputencoding{latin9}Theorem}

\newcommand{\symm}{\mathfrak{S}}

\newcommand\twoheaduparrow{\mathrel{\rotatebox{90}{$\twoheaduparrow$}}}
\newcommand\twoheaddownarrow{\mathrel{\rotatebox{270}{$\twoheaddownarrow$}}}

\begin{document}
	\author{Nick Early}
	\thanks{The author was partially supported by RTG grant NSF/DMS-1148634,\\
		University of Minnesota, email: \href{mailto:earlnick@gmail.com}{earlnick@gmail.com}}

	\title[A Canonical Permutohedral Plate Basis]{Canonical Bases for Permutohedral Plates}
	\maketitle
	\begin{abstract}

		We study three finite-dimensional quotient vector spaces constructed from the linear span of the set of characteristic functions of permutohedral cones by imposing two kinds of constraints: (1) neglect characteristic functions of higher codimension permutohedral cones, and (2) neglect characteristic functions of non-pointed permutohedral cones.  We construct an ordered basis which is canonical, in the sense that it has subsets which map onto ordered bases for the quotients.  We present straightening relations to the canonical basis, and using Laplace transforms we obtain functional representations for each quotient space.
		
		
	


	\end{abstract}
	
	\begingroup
	\let\cleardoublepage\relax
	\let\clearpage\relax
	\tableofcontents
	\endgroup
	
	\section{Introduction}

	The purpose of this paper is two-fold: it contains general combinatorial and geometric results about generalized permutohedra, but from our perspective it is motivated by surprising connections to physics, in particular to the study of scattering amplitudes in quantum field theory and string theory.

	This paper is devoted to the combinatorial analysis of the vector space of characteristic functions of permutohedral cones, studied as \textit{plates} by A. Ocneanu as communicated privately \cite{OcneanuCommunication}, and by the author in \cite{EarlyAnnouncement,EarlyAmplitudes2017}.  We derive a certain \textit{canonical} basis of the space which is spanned linearly by characteristic functions of permutohedral cones; these cones are \textit{dual} to the faces of the arrangement of reflection hyperplanes, defined by equations $x_i-x_j=0$.  
	
	The basis is called \textit{canonical} because of its compatibility with quotient maps to three other spaces: subsets of the canonical plate basis descend to bases for the quotients.  These maps are constructed from one or both of two geometrically-motivated conditions, according to which characteristic functions of (1) higher codimension faces, or (2) non-pointed cones, are sent to the zero element.

	Plates are \textit{permutohedral} cones: the edges extend along the root directions $e_i-e_j$.  Plates are labeled by ordered set partitions; when the blocks of an ordered partition are all singlets, then the corresponding plate is encoded by a directed tree of the form $\{(i_1,i_2),(i_2,i_3),\ldots, (i_{n-1},i_n)\}$ and correspondingly have edge directions the roots $e_{i_j}-e_{i_{j+1}}$.
	
	We begin in Section \ref{sec:Landscape} by introducing notation and collecting basic results about polyhedral cones.  In Section \ref{sec:platehomology} we give linear relations which express the characteristic function of any product of mutually orthogonal plates as a signed sum of plates which have the ambient dimension $n-1$.  We come to our main result, the construction of canonical plate basis, in Section \ref{sec: uppertriangularmap} and give two graded dimension formulas.  In Section \ref{sec:PlatesTrees} we give formulas which express the characteristic function of a permutohedral cone encoded by a directed tree as a signed sum of characteristic functions of plates.  In Theorem \ref{thm:Straightening} of Section \ref{sec:straightening}, we present a straightening formula which expands the characteristic function of any plate in the canonical basis.
	
	It would be interesting to study Hopf algebraic properties of the canonical plate basis.  For a possible starting point, one could look toward the product formula for WQSym, as represented on characteristic functions of polyhedral cones in Theorem 5.2 in \cite{ThibonMould}, keeping in mind the formula in Theorem \ref{thm: shuffleLump Identity} in the present paper.
	
	One consequence of Theorem \ref{thm:treeExpansionCompletePlates} and Corollary \ref{cor:plate relations quotient} is the possibility to derive the identities satisfied by the so-called Cayley functions from \cite{HeSon}; indeed, interpreting Cayley functions as Laplace transforms of permutohedral cones (in the terminology which follows, functional representations of type $\mathcal{P}_1^n$) solves the problem by allowing the application of Proposition \ref{prop:functional Tree Identities}.  By extension, this in part motivates the question studied in \cite{EarlyAmplitudes2017}, to explain in the same way the more intricate combinatorial constructions from \cite{Nonplanar, CachazoCombinatorialFactorization} using permutohedral cones.  Also relevant to the construction linking permutohedra and functional representations, are \cite{LoopParkeTaylor2,NewBCJ}.
	
	While the permutohedral cones studied in the present paper are all linear, having as faces hyperplanes which have the common intersection point $(0,\ldots, 0)$, one could introduce deformations, allowing translations by integer multiples of roots $e_i-e_j$.  This has some interesting consequences.  In the setting of \cite{EarlyAnnouncement}, such translations were essential in the proof of an $\symm_n$-equivariant analog of the classical Worpitzky identity from combinatorics: on one side of the identity are plates embedded in dilated simplices, and on the other side are plates intersected with translations of unit hypersimplices along integer multiples of roots $e_i-e_j$.  For a rather different multi-parameter deformation, for inspiration one could look toward the construction of the canonical form of the permutohedron in \cite{worldsheet}.

	\section{Landscape: basic properties of cones and plates}\label{sec:Landscape}
	Throughout this paper we shall assume $n\ge 1$.
	
	The all-subset hyperplane arrangement lives inside the linear hyperplane $V_0^n\subset\mathbb{R}^n$ defined by $x_1+\cdots+x_n=0$, and consists of the \textit{special} hyperplanes $\sum_{i\in I}x_i=0$, as $I$ runs through the proper nonempty subsets of $\{1,\ldots, n\}$.  This paper deals with the subspace of functions supported on $V_0^n$ lying in the span of characteristic functions $\lbrack U \rbrack$ of subsets $U$ which are closed cones in the all-subset arrangement.  Certain subsets $U$ play a special role in this theory.
	\begin{defn}\label{defn:plates}
		An \textit{ordered set partition} of $\{1,\ldots, n\}$ is a sequence of \textit{blocks} $(S_1,\ldots, S_k)$, where $\emptyset\not=S_i\subseteq\{1,\ldots,n\}$ with $S_i\cap S_j=\emptyset$ for $i\not=j$, and with $\cup_{i=1}^k S_i=\{1,\ldots, n\}$.  For each such ordered set partition, the associated \textit{plate} $\pi=\lbrack S_1,\ldots, S_k\rbrack$ is the cone defined by the system of inequalities
		\begin{eqnarray*}
			x_{S_1} & \ge & 0\\
			x_{S_1\cup S_2} & \ge & 0 \\
			& \vdots & \\  
			x_{S_1\cup\cdots\cup S_{k-1}} & \ge & 0\\
			x_{S_1\cup\cdots\cup S_{k}}=\sum_{i=1}^n x_i & =& 0,
		\end{eqnarray*}
		where $x_{S}=\sum_{i\in S}x_i$ for each $S\subseteq \{1,\ldots, n\}$.  When the $S_i$'s are singlets we write simply $\pi=\lbrack i_1,\ldots, i_n\rbrack$, where $i_a$ stands for the unique element of $S_a$.  We denote by $\lbrack\pi\rbrack=\lbrack\lbrack S_1,\ldots, S_k\rbrack\rbrack$ the characteristic function of the plate $\pi$.  
		
		Denote by $\text{len}(\pi)$ the number of blocks in the ordered set partition which labels $\pi$.
	\end{defn}

Note that the polyhedral cones in \cite{ThibonMould} are obtained from the above definition for permutohedral cones by relaxing the condition in the last line $\sum_{i=1}^n x_i=0$ to $\sum_{i=1}^n x_i\ge 0$.

\vspace{.5in}

	This paper studies four spaces $\hat{\mathcal{P}}^n, \mathcal{P}^n,\hat{\mathcal{P}}_1^n,\mathcal{P}_1^n$, as related by a diagram of linear surjections:
	\begin{eqnarray}\label{eqn: commutative diagram}
	\label{diagram}
	\begin{CD}
	\hat{\mathcal{P}}^n   @>{}>> \hat{\mathcal{P}}_1^n \\
	@V{}VV        @V{}VV\\
	\mathcal{P}^n   @>{}>> \mathcal{P}_1^n \\
	\end{CD}
	\end{eqnarray}
	The horizontal (respectively vertical) maps mod out by characteristic functions of cones which are not pointed (respectively not full-dimensional).  In particular, the upper left space $\hat{\mathcal{P}}^n$ is the linear span of characteristic functions of all plates.  The upper right space $\hat{\mathcal{P}}_1^n$ is the quotient of $\hat{\mathcal{P}}^n$ by the span of the characteristic functions of plates which are not pointed: they contain doubly infinite lines.  The lower left space $\mathcal{P}^n$ requires somewhat more care to define: it is the quotient of $\hat{\mathcal{P}}^n$ by the span of those linear combinations of characteristic functions of plates which vanish outside a set of measure zero in $V_0^n$.  The lower right space $\mathcal{P}_1^n$ is the quotient of $\hat{\mathcal{P}}^n$ by the span of both characteristic functions of non-pointed plates and linear combinations which vanish outside sets of measure zero in $V_0^n$.

	The subscript $1$ on $\hat{\mathcal{P}}_1^n$ and $\mathcal{P}_1^n$ is intended to remind that only characteristic functions of plates labeled by ordered set partitions having all blocks of size $1$ are nonzero.
	
	In Theorem \ref{thm:unitriangularMap} we construct the \textit{canonical} plate basis for $\hat{\mathcal{P}}^n$.  It is unitriangularly related to the given basis of characteristic functions $\lbrack\pi\rbrack$, with the virtue that it contains subsets which will descend to bases for the other three spaces.  In particular, this will show that the quotient spaces have dimensions given as follows:
	
	\begin{center}
		\begin{tabular}{c|c}
			$\dim\left(\hat{\mathcal{P}}^n\right)=\sum_{k=1}^n k!S(n,k)$ \ \ \ (Ordered Bell \#'s) & $\dim\left(\hat{\mathcal{P}}_1^n\right)\ =\ \sum_{k=1}^n s(n,k)=n!$ \\ 
			\hline 
			$\dim\left(\mathcal{P}^n\right)=\sum_{k=1}^n (k-1)!S(n,k)$ \ (Cyclic Bell \#'s) & $\ \dim\left(\mathcal{P}_1^n\right)\ =\ s(n,1)=(n-1)!$ \\ 
		\end{tabular} .
	\end{center}
	Here $S(n,k)$ is the $k^\text{th}$ Stirling number of the second kind, which counts the set partitions of $\{1,\ldots, n\}$ into $k$ parts, and $s(n,k)$ is the (unsigned) Stirling number of the first kind which counts the number of permutations which have $k$ cycles (including singlets) in their decompositions into disjoint cycles.
	
	The canonical basis of $\hat{\mathcal{P}}^n$ is naturally graded, and the corresponding dimensions are given below.  The first six rows are given below; note that the rows sum to the ordered Bell numbers $( 1, 3, 13, 75, 541, 4683)$.
	$$
	\begin{array}{cccccc}
	1&&&&&\\
	2 & 1 &  &  &  &\\
	6 & 6 & 1 &  &  &\\
	26 & 36 & 12 & 1 && \\
	150 & 250 & 120 & 20 & 1 & \\
	1082 & 2040 & 1230 & 300 & 30 & 1 \\
	\end{array}
	$$
	In Corollary \ref{cor:Lineardimensions} we observe that the rows are given by the equation 
	$$T_{n,k}=\sum_{i=k}^n S(n,i)s(i,k),$$
	where $S(n,i)$ is the Stirling number of the second kind, and $s(i,k)$ is the Stirling number of the first.  This is given in O.E.I.S. A079641.

	Further, as the canonical basis of $\hat{\mathcal{P}}^n$ passes to one for $\hat{\mathcal{P}}_1^n$, with graded dimensions the (unsigned) Stirling numbers of the first kind.
	$$
	\begin{array}{cccccc}
	1&&&&&\\
	1 & 1 &  &  &  &\\
	2 & 3 & 1 &  &  &\\
	6 & 11 & 6 & 1 && \\
	24 & 50 & 35 & 10 & 1 & \\
	120 & 274 & 225 & 85 & 15 & 1 \\
	\end{array}
	$$
	
	Finally, the space $\mathcal{P}_1^n$ has only one graded component, of dimension $(n-1)!$.  However, it turns out there is an algebra related to $\mathcal{P}_1^n$ which is the associated graded version of $\hat{\mathcal{P}}_1^n$, for which \cite{OrlikTerao} is relevant.  The study is beyond the scope of this paper and we leave it to future work.

	\subsection{Notation and conventions}

	Recall that a chain of inequalities $x_{i_1}\ge x_{i_2}\ge\cdots\ge x_{i_n}$ with $\sum x_i=0$ cuts out a simplicial cone called a \textit{Weyl chamber} in the arrangement of reflection hyperplanes, defined by equations $x_i-x_j=0$, of type $A_{n-1}$.  Weyl chambers are thus labeled by permutations $(i_1,\ldots, i_n)$ in one-line notation.
	
	A \textit{polyhedral cone} is an intersection of finitely many half spaces $\sum_{j=1}^n a_{i,j} x_j\ge 0$ in $\mathbb{R}^n$, for some integer constants $a_{i,j}$.  For example, the chain of inequalities which define a Weyl chamber can be reorganized as 
	$$x_{i_1}-x_{i_2}\ge 0,\ldots, x_{i_{n-1}}-x_{i_n}\ge 0,\ \text{ with } \sum x_i=0,$$
	so it is a polyhedral cone.  Letting $e_1,\ldots, e_n$ denote the standard basis for $\mathbb{R}^n$, if $I$ is a subset of $\{1,\ldots, n\}$, set $e_I=\sum_{i\in I} e_i$ and let $\bar{e}_I=\frac{\vert I^c\vert}{n}e_I-\frac{\vert I\vert}{n}e_{I^c}$ be the projection of $e_{I}$ onto $V_0^n$, along the vector $(1,\ldots, 1)$.  A polyhedral cone is \textit{pointed} if it does not contain any lines which extend to infinity in both directions.

	Given $\{v_1,\ldots, v_k\}\subset\mathbb{R}^n$, denote by $\langle v_1,\ldots, v_k\rangle_+:=\left\{c_1v_1+\cdots+c_kv_k: c_i\ge 0\right\}$ their \textit{conical hull}.  For example, a Weyl chamber is a conical hull, since it can also be obtained as the set of all linear combinations of a set of vectors with nonnegative coefficients,
	$$\left\langle\bar{e}_{i_1},\bar{e}_{\{i_1,i_2\}},\ldots, \bar{e}_{\{i_1,i_2,\ldots, i_{n-1}\}}\right\rangle_+=\{c_1\bar{e}_{i_1}+c_2\bar{e}_{\{i_1,i_2\}}+\cdots+c_{n-1}\bar{e}_{\{i_1,\ldots, i_{n-1}\}}:c_i\ge 0\}.$$
	Further, the conical hull of the \textit{roots} of type $A_{n-1}$,
	$\langle e_{i_1}-e_{i_2},\ldots, e_{i_{n-1}}-e_{i_n}\rangle_+$ is called a \textit{root cone}, and it is easy to check that it coincides with the plate $\lbrack i_1,\ldots, i_n\rbrack$.  We leave it as an exercise for the reader to check that, if $(S_1,\ldots, S_k)$ is an ordered set partition of $\{1,\ldots, n\}$, then we have
	\begin{eqnarray} 
	&& \lbrack S_1,\ldots, S_k\rbrack \nonumber\\
	&& =\left\langle e_a-e_b:\text{ either } a,b\in S_j \text{ for $j = 1,\ldots,k$, } \text{ or } (a,b)\in S_i\times S_{i+1} \text{ for $i=1,\ldots,k-1$} \right\rangle_+.\label{eqn: lumped plate conical hull}
	\end{eqnarray}
	
	\begin{defn}\label{defn:Minkowski sum}
		The \textit{Minkowski sum} of two polyhedral cones $C_1,C_2$ is given by $C_1+C_2=\left\{u+v:u\in C_1,\ v\in C_2\right\}$.  Denote by $\lbrack C_1\rbrack\cdot \lbrack C_2\rbrack=\lbrack C_1\cap C_2\rbrack$ the pointwise product of their characteristic functions, and denote by
		$$\lbrack C_1\rbrack\bullet \lbrack C_2\rbrack=\lbrack C_1+C_2\rbrack=\lbrack \left\{u+v:u\in C_1,\ v\in C_2\right\}\rbrack$$ 
		their \textit{convolution}, which is the characteristic function of the Minkowski sum $C_1+C_2$.
	\end{defn}
	
	\begin{rem}
		See \cite{BarvinokPammersheim} for details on the constructions of $\cdot$ and $\bullet$ as bi-linear maps on the $\mathbb{Q}$-vector space of characteristic functions of cones.  Note however that we extend the coefficient field to $\mathbb{C}$.
	\end{rem}

	There is a notion of duality for polyhedral cones.

	\begin{defn}\label{defn: duality for cones}
		Let $C$ be a polyhedral cone in $V_0^n$.  The dual cone to $C$, denoted $C^\star$, is defined by the equation 
		$$C^\star=\{y\in V_0^n:y\cdot x\ge 0\text{ for all } x\in C\}.$$
	\end{defn}

	\begin{rem}\label{rem: properties of dual cones}
		Dual cones are known to satisfy the following properties.
		\begin{enumerate}
			\item The dual $C^\star$ of a cone $C$ with nonempty interior is pointed.
			\item The dual $C^\star$ of a pointed cone $C$ has nonempty interior.
			\item If a cone $C$ is convex and topologically closed, then $(C^\star)^\star=C$.
		\end{enumerate}
		See for example \cite{ConvexOptimization} for details.
	\end{rem}

	\begin{rem}\label{rem: properties dual cones}
		We shall need the following results from \cite{BarvinokPammersheim}, Theorem 2.7 and respectively Corollary 2.8, where we extend the field from $\mathbb{Q}$ to $\mathbb{C}$.
		\begin{itemize}
			\item Duality for cones respects linear relations among their characteristic functions: if $C_1,\ldots, C_k$ are cones and for some constants $c_1,\ldots, c_k\in\mathbb{C}$ we have 
			$$\sum_{i=1}^kc_i\lbrack C_i\rbrack=0$$
			then the same relation holds among the characteristic functions for the dual cones,
			$$\sum_{i=1}^kc_i\lbrack C_i^\star\rbrack=0,$$
			interchanging the pointwise product and the convolution.
			\item There exists a linear map on the vector space spanned by characteristic functions of cones, which we also denote by $\star$, such that $\lbrack C\rbrack^\star=\lbrack C^\star\rbrack$.
			\item Duality for cones interchanges intersections and Minkowski sums: if $C_1,C_2$ are cones, then $\lbrack C_1\cap C_2\rbrack=\lbrack C_1\rbrack\cdot\lbrack C_2\rbrack$, and moreover 
			$$\left(\lbrack C_1\rbrack\cdot\lbrack C_2\rbrack\right)^\star=\lbrack C_1^\star\rbrack\bullet \lbrack C_2^\star\rbrack.$$
		\end{itemize}
	\end{rem}
	
	\begin{example}\label{example: Dual Plates}
		Denote $\bar{e}_1=(2,-1,-1)/3$ and $\bar{e}_2=(-1,2,-1)/3$, hence $\bar{e}_1+\bar{e}_2=(1,1,-2)/3=-\bar{e}_3$.  Then the two cones, the simple root cone
		$$\langle e_1-e_2,e_2-e_3\rangle_+$$
		and the Weyl chamber
		$$\langle\bar{e}_1,\bar{e}_1+\bar{e}_2\rangle_+,$$
		are dual to each other.  See Figure \ref{fig:dual-plates}.
	\end{example}
	\begin{figure}[h!]
		\centering
		\includegraphics[width=0.65\linewidth]{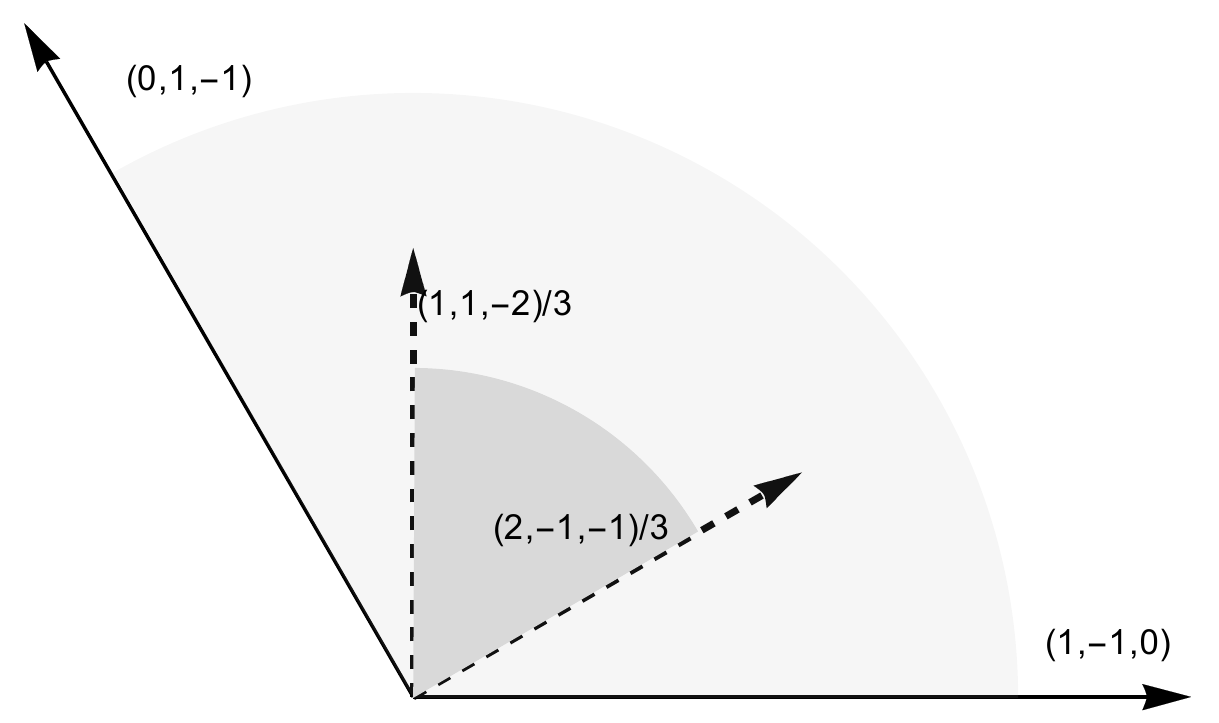}
		\caption{Dual Cones for Example \ref{example: Dual Plates}}
		\label{fig:dual-plates}
	\end{figure}

	\begin{defn}
		Denote by $\lbrack S_1,\ldots, S_k\rbrack^\star$ the face of the reflection arrangement labeled by the ordered set partition $(S_1,\ldots, S_k)$, given by the equations
		$$(x_{s_{1,1}}=\cdots=x_{s_{1,l_1}})\ge (x_{s_{2,1}}=\cdots=x_{s_{2,l_2}})\ge\cdots\ge(x_{s_{k,1}}=\cdots =x_{s_{k,l_k}}),\ \ x_1+\cdots+x_n=0$$
		(where $s_{i,1},\ldots, s_{i,l_i}$ are the elements of $S_i$), or for short
		$$x_{(S_1)}\ge x_{(S_2)}\ge \cdots\ge x_{(S_k)},\ x_1+\cdots+x_n=0,$$
		where for compactness the symbol $x_{(S)}$ is defined to be the list of equations $(x_{i_1}=\cdots=x_{i_{\vert S\vert}})$, for a subset $S=\{i_1,\ldots, i_{\vert S\vert}\}$ of $\{1,\ldots, n\}$.
		
		This can be given equivalently as the conical hull
		$$\left\langle \bar{e}_{S_1},\bar{e}_{S_1}+\bar{e}_{S_2},\ldots, \bar{e}_{S_1}+\cdots+\bar{e}_{S_{k-1}}\right\rangle_+.$$
	\end{defn}
	
	\begin{prop}\label{prop: dual higher codimension plate}
		If $(S_1,\ldots, S_k)$ is an ordered set partition of a subset $S\subseteq\{1,\ldots, n\}$, then the dual cone $\lbrack S_1,\ldots, S_k\rbrack^\star$ to
		$$\lbrack S_1,\ldots, S_k\rbrack=\left\{\sum_{i\in S}x_i e_i\in V_0^n: \sum_{i\in S_1\cup\cdots\cup S_j}x_i\ge 0,\ \text{for each } j=1,\ldots, k-1\right\}$$
		equals
		$$\left\{t_1\sum_{i\in S_1} e_i+\cdots+t_k\sum_{i\in S_k}e_i\in V_0^n:t_1\ge \cdots\ge t_k\right\}.$$
	\end{prop}
	
	\begin{proof}
		Suppose $y\cdot x\ge 0$ for all $x\in \lbrack S_1,\ldots, S_k\rbrack$.  By Equation \eqref{eqn: lumped plate conical hull},
		$$\lbrack S_1,\ldots, S_k\rbrack=\left\langle e_a-e_b:\text{ either } a,b\in S_i \text{ or $(a,b)\in S_i\times S_{i+1}$ for $i=1,\ldots,k-1$ } \right\rangle_+,$$
		it suffices to check that, for either $a,b\in S_i$ or $(a,b)\in S_i\times S_{i+1}$, we have
		$$y_a-y_b=y\cdot (e_a-e_b)\ge0\text{ and } y_b-y_a= y\cdot (e_b-e_a)\ge 0.$$
		Thus, 
		$$y=t_1\sum_{i\in S_1}e_i+\cdots+t_k\sum_{i\in S_k}e_i+\sum_{i\in S^c} y_i e_i$$
		for some $t_1,\ldots, t_k$ and $y_i$, such that $\vert S_1\vert t_1+\cdots+\vert S_k\vert t_k+\sum_{i\in S^c} y_i=0$.  Now, for each $i=1,\ldots, k-1$, for any $(a,b)\in S_i\times S_{i+1}$ we have $t_i-t_{i+1}=y\cdot (e_a-e_b)\ge 0$, or $t_i\ge t_{i+1}$.
	\end{proof}
	
	Corollary \ref{cor: reduction to Minkowski sum} follows from \eqref{eqn: lumped plate conical hull}.  For pedagogical reasons we include a different proof using duality.
	\begin{cor}\label{cor: reduction to Minkowski sum}
		If $\pi=\lbrack S_1,\ldots, S_k\rbrack$ is a plate, then we may factor its characteristic function using the pointwise product.  Set theoretically we have
		$$\pi=\lbrack S_1,\ldots, S_k\rbrack=\bigcap_{i=1}^{k-1}\lbrack S_1\cup\cdots\cup S_i,S_{i+1}\cup\cdots\cup S_k\rbrack,$$
		or in terms of the pointwise product for characteristic functions,
		$$\lbrack\pi\rbrack=\lbrack\lbrack S_1,S_2\cup S_3\cdots \cup S_k\rbrack\rbrack\cdot\lbrack\lbrack S_1\cup S_2, S_3\cup \cdots \cup S_k\rbrack\rbrack\cdots \lbrack\lbrack S_1\cup S_2 \cdots \cup S_{k-1},S_k\rbrack\rbrack.$$
		Using Minkowski sums it can be expressed as
		$$\pi=\sum_{i=1}^{k-1}\lbrack S_i,S_{i+1}\rbrack,$$
		and correspondingly terms of the convolution product for characteristic functions,
		$$\lbrack\pi\rbrack=\lbrack\lbrack S_1,S_2\rbrack\rbrack\bullet\lbrack\lbrack S_2,S_3\rbrack\rbrack\bullet\cdots \bullet \lbrack\lbrack S_{k-1},S_k\rbrack\rbrack,$$
		where 
		$$\lbrack S_i,S_{i+1}\rbrack=\left\{\sum_{j\in S_i\cup S_{i+1}}x_j e_j\in V_0^n:\sum_{j\in S_i}x_j\ge 0,\  \sum_{j\in S_i\cup S_{i+1}}x_{j}=0\right\}.$$
	\end{cor}
	
	\begin{proof}
		The identity for the intersection products follows immediately from the defining inequalities for plates.
		
		For the convolution product identity it is convenient first to dualize,
		\begin{eqnarray*}
			\lbrack\lbrack S_1,S_2\rbrack\rbrack\bullet\lbrack\lbrack S_2,S_3\rbrack\rbrack\bullet\cdots \bullet \lbrack\lbrack S_{k-1},S_k\rbrack\rbrack & = & \left(\lbrack\lbrack S_1,S_2\rbrack\rbrack^\star\cdot\lbrack\lbrack S_2,S_3\rbrack\rbrack^\star\cdots  \lbrack\lbrack S_{k-1},S_k\rbrack\rbrack^\star\right)^\star,
		\end{eqnarray*}
		where by Proposition \ref{prop: dual higher codimension plate} we have
		$$\lbrack\lbrack S_{i-1},S_i\rbrack\rbrack^\star =\left\lbrack\{y\in V_0^n:y_{(S_{i})}\ge y_{(S_{i+1})} \}\right\rbrack.$$
		
		Thus,
		\begin{eqnarray*}
			& & \left(\lbrack\lbrack S_1,S_2\rbrack\rbrack^\star\cdot\lbrack\lbrack S_2,S_3\rbrack\rbrack^\star\cdots  \lbrack\lbrack S_{k-1},S_k\rbrack\rbrack^\star\right)^\star\\
			& = & \left(\left\lbrack\{x\in V_0^n:x_{(S_{1})}\ge x_{(S_{2})} \}\right\rbrack\cdots \left\lbrack\{x\in V_0^n:x_{(S_{k-1})}\ge x_{(S_{k})}\}\right\rbrack\right)^\star\\
			& =& \left\lbrack\{x\in V_0^n:x_{(S_{1})}\ge\cdots\ge  x_{(S_k)}\}\right\rbrack^\star\\
			& = & \left\lbrack\left\{x\in V_0^n:x_{S_1}\ge 0,\ldots, \ x_{S_1\cup \cdots\cup S_{k-1}}\ge0\right\}\right\rbrack\\
			& = & \lbrack\lbrack S_1,\ldots, S_k\rbrack\rbrack,
		\end{eqnarray*}
		where as usual we use the shorthand notations $x_S=\sum_{i\in S}x_i$ and $x_{(S)}=(x_{s_1}=\cdots=x_{s_{\vert S\vert}})$, for $S$ a subset of $\{1,\ldots, n\}$.

	\end{proof}

	\subsection{Genericity}
	The relationships between $\hat{\mathcal{P}}^n,\mathcal{P}^n,\hat{\mathcal{P}}_1^n$ and $\mathcal{P}_1^n$ reduce to variations on a single linear identity.  If $S_1,S_2\subsetneq\{1,\ldots, n\}$ are any two disjoint nonempty subsets, then the characteristic function of the set 
	$$\{x\in V_0^n:x_{S_1}=x_{S_2}=0\text{ and } x_j=0\text{ for all }j\not\in S_1\cup S_2\},$$ represented by the pointwise product $\lbrack\lbrack S_1,S_2\rbrack\rbrack\cdot \lbrack\lbrack S_2,S_1\rbrack\rbrack$, can be expressed via
	$$\lbrack\lbrack S_1,S_2\rbrack\rbrack+\lbrack\lbrack S_2,S_1\rbrack\rbrack=\lbrack\lbrack S_1,S_2\rbrack\rbrack\cdot \lbrack\lbrack S_2,S_1\rbrack\rbrack+\lbrack\lbrack S_1\cup S_2\rbrack\rbrack.$$
	The above identity holds in $\hat{\mathcal{P}}^n$.  To obtain the identity which holds in $\mathcal{P}^n$ we specialize to 
	$$\lbrack\lbrack S_1, S_2\rbrack\rbrack+\lbrack\lbrack S_2,S_1\rbrack\rbrack=\lbrack\lbrack S_1\cup S_2\rbrack\rbrack,$$
	where $\lbrack\lbrack S_1\cup S_2\rbrack\rbrack$ is the characteristic function of the set 
	$$\lbrack S_1\cup S_2\rbrack = \left\{\sum_{j\in S_1\cup S_2} x_j e_j\in \mathbb{R}^n:\sum_{j\in S_1\cup S_2} x_j=0\right\}\subseteq V_0^n.$$
	In $\hat{\mathcal{P}}_1^n$ and $\mathcal{P}_1^n$ we we assume that $S_1,S_2\subset\{1,\ldots, n\}$ are singlets and specialize to respectively
	$$\lbrack\lbrack S_1,S_2\rbrack\rbrack+\lbrack\lbrack S_2,S_1\rbrack\rbrack =\lbrack\lbrack S_1,S_2\rbrack\rbrack\cdot\lbrack\lbrack S_2,S_1\rbrack\rbrack$$
	and
	$$\lbrack\lbrack S,S^c\rbrack\rbrack+\lbrack\lbrack S^c,S\rbrack\rbrack =0.$$

	\section{Plate homology: from plates to their faces }\label{sec:platehomology}
	The main result of this section is Theorem \ref{thm: shuffleLump Identity}, which expands the face of a plate as a linear combination of characteristic functions of plates in $\hat{\mathcal{P}}^n$.
	
	While the symmetric group does not play an essential role in this paper, let us point out some of the symmetry properties of plates which were implicit in \cite{EarlyAnnouncement}.  The action of the symmetric group $\symm_n$ on plates is inherited from the coordinate permutation on $\mathbb{R}^n$.  In the plate notation, $\sigma\in\symm_n$ acts on characteristic functions of plates $\lbrack\lbrack S_1,\ldots, S_k\rbrack\rbrack$ by permuting elements in the blocks $S_i$.
	
	\begin{rem}
		The permutation group $\symm_n$ preserves the following operations on characteristic functions of plates.
		\begin{enumerate}
			\item Inclusion: $$\lbrack\lbrack S_1,\ldots, S_k\rbrack\rbrack\mapsto \lbrack\lbrack T_1,\ldots, T_l\rbrack\rbrack$$
			where $l\le k$, each $T_i$ is a union of some consecutive $S_j$'s and we still have $\cup_{i=1}^l T_i=\{1,\ldots, n\}$.  Note that from the defining inequalities we have the inclusion of cones
			$$\lbrack S_1,\ldots, S_k\rbrack\subseteq \lbrack T_1,\ldots, T_l\rbrack.$$
			\item Block permutation: $$\lbrack\lbrack S_1,\ldots, S_k\rbrack\rbrack\mapsto \lbrack\lbrack S_{\tau_1},\ldots, S_{\tau_k}\rbrack\rbrack$$
			for a permutation $\tau\in \mathfrak{S}_k$.
			\item Restriction to a face: 
			$$\lbrack\lbrack S_1,\ldots, S_k\rbrack\rbrack=\lbrack\lbrack S_1,S_2\rbrack\rbrack\bullet\cdots\bullet\lbrack\lbrack S_{k-1}, S_k\rbrack\rbrack \mapsto \lbrack\lbrack S_{i_1},S_{i_1+1}\rbrack\rbrack\bullet\cdots\bullet\lbrack\lbrack S_{i_{l}}, S_{i_{l}+1}\rbrack\rbrack$$
			for any subset $\{i_1,\ldots, i_l\}\subseteq\{1,\ldots, k-1\}$.
		\end{enumerate}
	\end{rem}
	

	\begin{prop}\label{Prop: UniversalPlateModuleFreelyGenerated}
		The space $\hat{\mathcal{P}}^n$, the linear span of the characteristic functions of plates $\pi$, has linear dimension the ordered Bell number $\sum_{k=1}^{n} k!S_{n,k}$, where $S_{n,k}$ are the Stirling numbers of the second kind, which count the number of set partitions of $\{1,\ldots, n\}$ into $k$ disjoint subsets.
	\end{prop}
	
	\begin{proof}
		By Remark \ref{rem: properties dual cones}, the involution $\star$ preserves linear relations among characteristic functions; therefore it provides a natural isomorphism of vector spaces from $\hat{\mathcal{P}}^n$ onto the space spanned by the characteristic functions of faces of Weyl chambers.  Further, for each dimension $k=0,\ldots, n-1$, these faces have non-intersecting relative interiors and consequently their characteristic functions are linearly independent, and by the duality map $\star$, plates in $\hat{\mathcal{P}}^n$ are as well.  Therefore, to extract the dimension formula it suffices to count the faces of the arrangement of reflection hyperplanes; but these are in bijection with the ordered set partitions of $\{1,\ldots, n\}$, which are counted by the ordered Bell numbers.
	\end{proof}

	\begin{defn}\label{defn:lumped permutation, etc}
		Let $\mathbf{S}=(S_1,\ldots, S_k)$ be an ordered set partition of $\{1,\ldots, n\}$.  Suppose $\mathbf{T}=(T_1,\ldots, T_l)$ is another ordered set partition of $\{1,\ldots, n\}$, such that each block $T_i$ of $\mathbf{T}$ is a union of blocks $S_j$ of $\mathbf{S}$.  Define $p_i\in\{1,\ldots, l\}$ by the condition that $S_i$ is a subset of $T_{p_i}$.  Thus if $S_i$ is a subset of the first block of $\mathbf{T}$ then $p_1=1$.  Then $\mathbf{T}$ has the \textit{orientation} $p_i\le p_j$ with respect to $\mathbf{S}$ whenever $S_i$ appears to the left of $S_j$ in $\mathbf{T}$, with equality if and only if $S_i\sqcup S_j\subseteq T_{a}$ for some $a\in\{1,\ldots, l\}$.  In the case that $\mathbf{T}$ satisfies $p_i<p_j$, we shall say that $T$ is \textit{compatible} with the orientation $p_i<p_j$.
		
	\end{defn}
	
	\begin{defn}\label{defn:shuffle-Lumping}
		Let $\mathbf{S}=(S_1,\ldots, S_k)$ be an ordered set partition of $\{1,\ldots, n\}$		Another ordered set partition $\mathbf{T}=(T_1,\ldots, T_m)$ of $\{1,\ldots, n\}$ is a \textit{shuffle-lumping} of ordered set partitions 
		$$\mathbf{S}_1=(S_{1},S_{2},\ldots, S_{k_1}), \mathbf{S}_2=(S_{k_1+1},S_{k_1+2},\ldots, S_{k_1+k_2}),\ldots, \mathbf{S}_l=(S_{k_1+\cdots+k_{l-1}+1},\ldots, S_{k_1+\cdots+k_{l}}),$$
		provided that each ordered set partition $\mathbf{S}_i$ is compatible with the orientations
		$$p_{1}<p_2<\cdots<p_{k_1}$$
		$$p_{k_1+1}<p_{k_1+2}<\cdots<p_{k_1+k_2}$$
		$$\vdots$$
		$$p_{k_1+\cdots+k_{l-1}+1}<p_{k_1+\cdots+k_{l-1}+2}<\cdots<p_n.$$
	\end{defn}
	\begin{example}
		The plate $\lbrack S_1,\ldots, S_k\rbrack$ is uniquely characterized among its shuffle-lumpings by the set of orientations $p_1< p_2<\cdots < p_k$ on the blocks $S_1,\ldots, S_k$.
	\end{example}

	\begin{example}
		If $(S_1,S_2,S_3,S_4)=(1,4,23,5)$ and $(S_5,S_6)=(678,9)$, then shuffle-lumped plates include for example
		$$\lbrack1,4678,23,59\rbrack\text{ and } \lbrack678,1,4,23,9,5\rbrack.$$
		Then $\lbrack1,4678,23,59\rbrack$ is a plate with the smallest possible number of blocks, while $\lbrack678,1,4,23,9,5\rbrack$ is a plate with the largest possible number of blocks, in the shuffle-lumping of the set compositions $(S_1,S_2,S_3,S_4)$ and $(S_5,S_6)$.
	\end{example}

	\begin{example}
		The shuffle-lumpings of $(S_1,S_2)=(\{1\},\{2,3\})$ and $(S_3,S_4)=(\{4\},\{5\})$ are
		$$\{(1,23,4,5),(1,234,5),(1,4,23,5),(14,23,5),(4,1,23,5),(1,4,235),(14,235),(4,1,235),$$
		$$(1,4,5,23),(14,5,23),(4,1,5,23),(4,15,23),(4,5,1,23)\}.$$
	\end{example}
	
	In Lemma \ref{lem: Weyl chamber decomposition shuffle} we decompose the characteristic function of a union of closed Weyl chambers into an alternating sum of partially closed Weyl chambers in a canonical way that depends on descent positions, with respect to the natural order $(1,\ldots, n)$.  See \cite{Stanley EC1} for a systematic approach using so-called $(P,\omega)$-partitions.

	\begin{lem}\label{lem: Weyl chamber decomposition shuffle}
		We have the decomposition into disjoint sets
		$$V_0^n=\sqcup_{\sigma\in \symm_n} C_\sigma,$$
		where each $C_\sigma$ is the partially open Weyl chamber defined by 
		$$x_{\sigma_i}\ge x_{\sigma_{i+1}}\ \text{ if } \sigma_{i}<\sigma_{i+1}$$
		and
		$$x_{\sigma_i}>x_{\sigma_{i+1}}\ \text{ if } \sigma_{i}>\sigma_{i+1}.$$
		The characteristic function of $C_\sigma$, with $\sigma=(\sigma_1,\ldots, \sigma_n)$ given, is
		$$\lbrack C_\sigma\rbrack=\sum_{\pi}(-1)^{n-\operatorname{len}(\pi)}\lbrack\pi^\star\rbrack,$$
		where the sum is over the set of plates $\pi=\lbrack(S_1,\ldots, S_k)\rbrack$ which are labeled by ordered set partitions $(S_1,\ldots, S_k)$ with blocks $S_i$ defined as follows.  Let 
		$$\{d_1<\cdots<d_{k-1}\}=\{i: \sigma_i <\sigma_{i+1}\}.$$
		Also set $d_0=0$ and $d_k=n$.  Then put
		$$S_i = \{\sigma_p:d_{i-1}<p\le d_i\}.$$
		

	\end{lem}

	\begin{proof}
		Any $x\in V_0^n$ not in \textit{any} reflection hyperplane $x_i=x_j$ is in the interior of the Weyl chamber labeled by the order of its coordinate values, $x_{\sigma_1}> \cdots> x_{\sigma_n}$, say.  Now if $x$ is in the interior of a face labeled by an ordered set partition $(S_1,\ldots, S_k)$ of $\{1,\ldots, n\}$, having the form
		$$x=t_1 e_{S_1}+\cdots+t_k e_{S_k},$$
		for some $t_1>t_2>\cdots>t_k$ with $\sum_{i=1}^k\vert S_i\vert t_i=0$, then we put $x\in C_\sigma$ where the permutation $\sigma$ is obtained from $(S_1,\ldots, S_k)$ by placing the labels in each block $S_i$ in increasing order and then concatenating the blocks.

		The formula for $\lbrack C_\sigma\rbrack$ follows from the standard inclusion-exclusion expression for the characteristic function of the complement of the union of the codimension 1 faces $\lbrack \sigma_1,\ldots, \sigma_{i}\sigma_{i+1},\ldots,\sigma_n\rbrack^\star$ corresponding to descents $\sigma_i>\sigma_{i+1}$ in $\sigma$:
		$$\lbrack\sigma_{1},\ldots,\sigma_n\rbrack^\star\setminus\left(\bigcup_{\sigma_i>\sigma_{i+1}}\lbrack \sigma_1,\ldots, \sigma_{i}\sigma_{i+1},\ldots,\sigma_n\rbrack^\star\right),$$
		that is
		$$\lbrack C_\sigma\rbrack=\left\lbrack \lbrack\sigma_{1},\ldots,\sigma_n\rbrack^\star\right\rbrack-\sum_{\pi'}(-1)^{n-1-\operatorname{len}(\pi')}\lbrack(\pi')^\star\rbrack=\sum_{\pi}(-1)^{n-\operatorname{len}(\pi)}\lbrack\pi^\star\rbrack$$
		where the middle sum is over the lumpings $\pi'$ of $\lbrack\sigma_1,\ldots,\sigma_n\rbrack$ at descents $\sigma_i>\sigma_{i+1}$ (excluding $\lbrack\sigma_1,\ldots,\sigma_n\rbrack$ itself), and the right sum now includes $\lbrack\sigma_1,\ldots,\sigma_n\rbrack$.
	\end{proof}
	
	\begin{example}
		The characteristic functions of the partially open Weyl chambers respectively
		$$\{x\in V_0^3: x_1\ge x_2\ge x_3\}$$
		$$\{x\in V_0^3: x_2> x_1\ge x_3\}$$
		$$\{x\in V_0^3: x_3> x_2> x_1\}$$
		can be obtained as linear combinations of characteristic functions of dual plates as
		\begin{eqnarray*}
			\lbrack C_{(1,2,3)}\rbrack & = & \lbrack\lbrack 1,2,3\rbrack\rbrack^\star\\
			\lbrack C_{(2,1,3)}\rbrack & = & \lbrack \lbrack 2,1,3\rbrack\rbrack^\star -\lbrack\lbrack 21,3\rbrack\rbrack^\star\\
			\lbrack C_{(3,2,1)}\rbrack & = & \lbrack \lbrack 3,2,1\rbrack\rbrack^\star -\lbrack\lbrack 32,1\rbrack\rbrack^\star-\lbrack\lbrack 3,21\rbrack\rbrack^\star+\lbrack\lbrack 321\rbrack\rbrack^\star.\\
		\end{eqnarray*}
		
	\end{example}
	
	More generally, in Lemma \ref{lem: Weyl chamber decomposition shuffle} we replace Weyl chambers, labeled by permutations, with higher codimension faces of the reflection arrangement which are labeled by ordered set partitions.  Suppose 
	$$\mathbf{S}_1=(S_1,\ldots, S_{k_1}),\ldots,\mathbf{S}_l=(S_{k_1+\cdots+k_{l-1}+1},\ldots, S_{k_1+\cdots+k_{l}})$$
	are ordered set partitions of respectively $\bigcup_{S\in\mathbf{S}_i}S$, $i=1,\ldots, l$, and let $\sigma=(\sigma_1,\ldots,\sigma_l)$ be a permutation of $\{1,\ldots, l\}$.  Define an embedding $\iota:V_0^l\hookrightarrow V_0^{k_1+\cdots +k_l}$ by
	$$\sum_{i=1}^lt_i e_i\mapsto \sum_{i=1}^l \left(\frac{t_i}{\vert S_i\vert}\right) e_{S_i}.$$
	\begin{cor}\label{cor: Higher codimension Weyl chamber decomposition shuffle}
		We have
		$$\lbrack\iota(C_{\sigma})\rbrack=\sum_{\pi}(-1)^{l-\operatorname{len}(\pi)}\lbrack\iota(\pi^\star)\rbrack,$$
		where the sum is the same as in Lemma \ref{lem: Weyl chamber decomposition shuffle}.
	\end{cor}

	In Theorem \ref{thm: shuffleLump Identity}, we replace the natural order $(1,\ldots, m)$ with an ordered set partition $$(S_{1},S_2,\ldots, S_{m})$$ of $\{1,\ldots, n\}$, where $m=k_1+\cdots+k_l$. 
	
	The proof in what follows of Theorem \ref{thm: shuffleLump Identity} illustrates the essential role of the duality isomorphism from Definition \ref{defn: duality for cones} and utilizes directly set-theoretic inclusion-exclusion arguments.  Now, by way of Theorem 5.2 of \cite{ThibonMould}, the same formula holds if we extend plates from $V_0^n$ into the ambient space $\mathbb{R}^n$, in which case the last line of the plate equations becomes $x_1+\cdots+x_n\ge 0$.
	
	\begin{thm}\label{thm: shuffleLump Identity}
		Given $l$ ordered set partitions
		$$\mathbf{S}_1=(S_{1},S_{2},\ldots, S_{k_1}),\mathbf{S}_2=(S_{k_1+1},S_{k_1+2},\ldots, S_{k_1+k_2}),\ldots, \mathbf{S}_{l}=(S_{k_1+\cdots+k_{l-1}+1},\ldots, S_{k_1+\cdots+k_{l}})$$
		such that $\bigsqcup_{i=1}^{k_1+\cdots+k_l}S_i=\{1,\ldots, n\}$, then we have the identity for characteristic functions of plates in $\hat{\mathcal{P}}^n$,
		
		$$\lbrack\lbrack\mathbf{S}_1\rbrack\rbrack\bullet \cdots\bullet \lbrack\lbrack \mathbf{S}_{l}\rbrack\rbrack= \sum_\pi(-1)^{m-\operatorname{len}(\pi)}\lbrack\pi\rbrack,$$
		where $m=k_1+\cdots+k_l$ and $\pi$ runs over all shuffle-lumpings of $\mathbf{S}_1,\ldots,\mathbf{S}_l$.
	\end{thm}

	\begin{proof}
		We shall work in the space of characteristic functions of faces of the reflection arrangement and then dualize to obtain the identity for characteristic functions of plates.  
		
		We have 
		\begin{eqnarray*}\label{eqn:weight plates for shuffleLump identity}
			\lbrack\mathbf{S}_i\rbrack^\star & = & \left\{\sum_{i=1}^{m}t_ie_{S_i}\in V_0^n:t_{k_{i-1}+1}\ge \cdots\ge t_{k_{i}}\right\}\\
		\end{eqnarray*}
		and thus
		\begin{eqnarray*}
			\lbrack \mathbf{S}_1\rbrack^\star \cap\cdots\cap\lbrack \mathbf{S}_l\rbrack^\star & = & \left\{\sum_{i=1}^{m}t_ie_{S_i}:\begin{array}{c}
				t_1\ge t_2\ge \cdots\ge t_{k_1}	, 	\\ 
				t_{k_1+1}\ge\cdots\ge t_{k_1+k_2}, 	\\ 
				\vdots							\\
				t_{k_1+\cdots+k_{l-1}+1}\ge\cdots\ge t_{k_1+\cdots+k_l}
			\end{array} \right\},\\
		\end{eqnarray*}
		which lives in a copy of $V_0^m$ embedded in $V_0^n$ as $\iota(\sum_{i=1}^{m}t_ie_i) \mapsto \sum_{i=1}^{m}t_i (e_{S_i}/\vert S_i\vert)$.
		
		Then $\iota^{-1}\left(\lbrack \mathbf{S}_1\rbrack^\star \cap\cdots\cap\lbrack \mathbf{S}_l\rbrack^\star\right)\subseteq V_0^m$ is a union of Weyl chambers $\bigcup_{\tau}\lbrack\tau\rbrack^\star$
		defined by $y_{\tau_1}\ge\cdots\ge y_{\tau_m}$ labeled by shuffles $(\tau_1,\ldots,\tau_m)$ of
		$$\sigma_1=(1,2,\ldots, k_1),\ldots, \sigma_l=\left(m-k_l+1,\ldots, m\right).$$  
		We replace each such (closed) Weyl chamber defined by $y_{\tau_1}\ge\cdots\ge y_{\tau_m}$ with the partially open Weyl chamber $C_\tau$ from Lemma \ref{lem: Weyl chamber decomposition shuffle} and obtain the disjoint union 
		\begin{eqnarray*}
			\iota^{-1}\left(\lbrack \mathbf{S}_1\rbrack^\star \cap\cdots\cap\lbrack \mathbf{S}_l\rbrack^\star\right) & \supseteq & \bigsqcup_{\tau}C_\tau,
		\end{eqnarray*}
		where the disjoint union is dense in the (closed) left hand side.  Thus, equality will follow once we establish that $\bigsqcup_{\tau}C_\tau$ is already topologically closed.  
		
		Supposing $x$ is in a missing boundary face of some partially open Weyl chamber $C_\tau$, then the coordinates of $x$ satisfy an equality $x_d=x_{d+1}$ where $\tau_d>\tau_{d+1}$ is a descent of $\tau$.  But since $\sigma_1,\ldots, \sigma_l$ are all increasing, this can happen only if $\tau_d$ and $\tau_{d+1}$ belong to two different permutations, say $\sigma_i$ and respectively $\sigma_j$, for some $i\not=j$.  This implies that the permutation $\tau'$ obtained from $\tau$ by switching $\tau_d$ and $\tau_{d+1}$ is also a shuffle of $\sigma_1,\ldots, \sigma_l$, hence $x\in C_{\tau'}$, proving the equality.
		
		This, together with the expansion from Lemma \ref{lem: Weyl chamber decomposition shuffle} for the characteristic function $\lbrack C_\tau\rbrack$ implies the identity for characteristic functions
		$$\lbrack\iota^{-1}\left(\lbrack \mathbf{S}_1\rbrack^\star \cap\cdots\cap\lbrack \mathbf{S}_l\rbrack^\star\right)\rbrack=\sum_{\tau}\lbrack C_\tau\rbrack=\sum_{\tau}\left(\sum_{\pi_\tau}(-1)^{m-\operatorname{len}(\pi_\tau)}\lbrack\iota^{-1}(\pi_\tau^\star)\rbrack\right),$$
		where the inner sum is over all lumpings of the plate $\pi_\tau$ which can occur at the descents of $\tau$.
		
		It follows from Theorem 2.3 of \cite{BarvinokPammersheim} that the $\iota$ induces a unique linear map on the space of characteristic functions, and we obtain
		$$\lbrack\lbrack \mathbf{S}_1\rbrack^\star \cap\cdots\cap\lbrack \mathbf{S}_l\rbrack^\star\rbrack=\sum_{\tau}\lbrack \iota(C_\tau)\rbrack=\sum_{\tau}\left(\sum_{\pi_\tau}(-1)^{m-\operatorname{len}(\pi_\tau)}\lbrack\pi_\tau^\star\rbrack\right).$$
		We finally dualize again to obtain the sum over all shuffle-lumpings
		$$\lbrack\lbrack \mathbf{S}_1\rbrack\rbrack\bullet \cdots \bullet\lbrack \lbrack\mathbf{S}_l\rbrack\rbrack=\sum_\pi(-1)^{m-\operatorname{len}(\pi)}\lbrack\pi\rbrack.$$

	\end{proof}
	
	Note that in the convolution $\lbrack\lbrack\mathbf{S}_1\rbrack\rbrack \bullet\cdots\bullet\lbrack\lbrack\mathbf{S}_l\rbrack\rbrack$, since the subsets $\cup_{S\in \mathbf{S}_1}S,\ldots, \sum_{S\in\mathbf{S}_l}S$ are assumed to be mutually disjoint, the plates $\lbrack\mathbf{S}_i\rbrack$ live in mutually orthogonal subspaces 
	$$\left\{x\in V_0^n:\sum_{S\in\mathbf{S}_i}x_S =0,\text{ and } x_j=0\text{ for } j\not\in\cup_{S\in\mathbf{S}_i}S\right\}\subseteq V_0^n,$$
	for $i=1,\ldots, l$.
	\begin{example}
		Let $S_{1}=\{1\}$ and $(S_{2},S_{3})=(\{2\},\{3\})$.  Then Theorem \ref{thm: shuffleLump Identity} says that
		$$\lbrack\lbrack 1\rbrack\rbrack\bullet \lbrack\lbrack 2,3\rbrack\rbrack =\lbrack\lbrack 1,2,3\rbrack\rbrack+\lbrack\lbrack 2,1,3\rbrack\rbrack+\lbrack\lbrack 2,3,1\rbrack\rbrack-\left(\lbrack\lbrack 12,3\rbrack\rbrack+\lbrack\lbrack 2,13\rbrack\rbrack\right).$$
		See Figure \ref{fig:plate-relations-complete}.
	\end{example}
	
	\begin{figure}[h!]
		\centering
		\includegraphics[width=.80\linewidth]{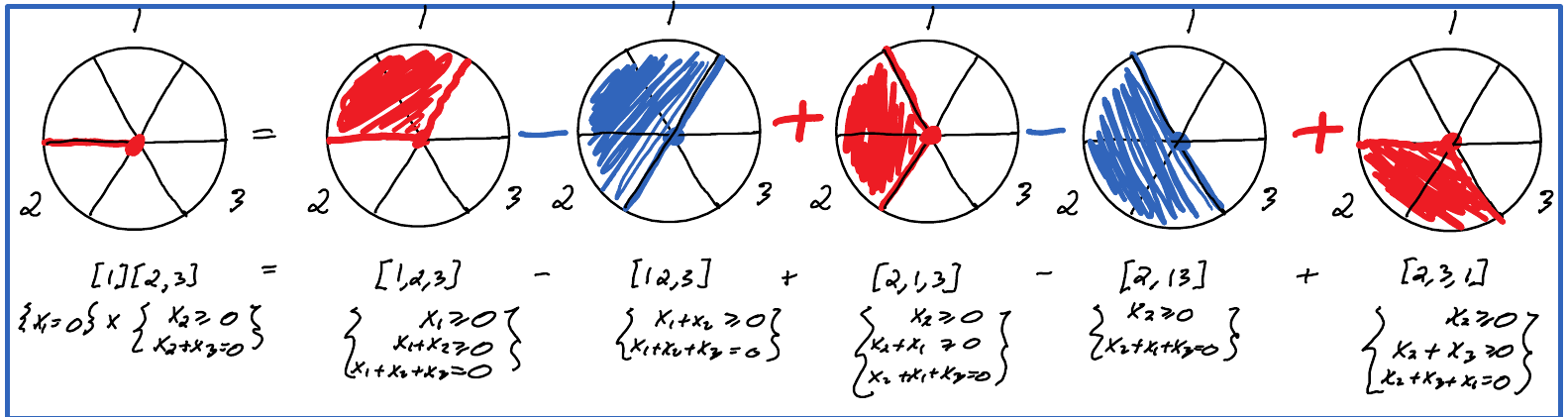}
		\caption{The characteristic function $\lbrack\lbrack 1\rbrack\rbrack\bullet \lbrack\lbrack 2,3\rbrack\rbrack$}
		\label{fig:plate-relations-complete}
	\end{figure}
	
	\section{Constructing the canonical plate basis}\label{sec: uppertriangularmap}

	We generalize the notion of the cycle decomposition, from permutations to ordered set partitions.  The geometric motivation is to establish a graded basis for $\hat{\mathcal{P}}^n$ such that the $d^\text{th}$ graded piece is spanned by characteristic functions of faces of dimension $d$ of the all-subset hyperplane arrangement.  These faces are Minkowski sums of standard plates which lie in mutually orthogonal subspaces of $V_0^n$.
	
	In what follows, we fix once and for all the \textit{standard} ordered set partition $(I_1,I_2,\dots, I_n)$ of $\{1,\ldots, n\}$, where $I_j=\{j\}$.  For compactness, we shall abuse notation and write $j$ instead of $I_j$.  However, it should not be forgotten that this obscures an action of the product group $\mathfrak{S}_n\times \mathfrak{S}_n$, where one factor permutes the order of the blocks and the other permutes their contents.

	\begin{defn}\label{defn:composite set partition}
		A \textit{composite set partition} of $\{1,\ldots, n\}$ is a set $\{\mathbf{S}_1,\ldots,\mathbf{S}_l\}$ where each $\mathbf{S}_i$ is an ordered set partition of a subset $J_i\subset \{1,\ldots, n\}$, such that $\{J_1,\ldots, J_l\}$ is an (unordered) set partition of $\{1,\ldots, n\}$.  If each $\mathbf{S}_i$ has the property that its first block contains the minimal label in $J_i$, then the composite set partition is called \textit{standard.}
	\end{defn}
Say that a composite set partition $\{\mathbf{S}_1,\ldots,\mathbf{S}_l\}$ satisfies the orientations, $p_a<p_b$ respectively $p_a>p_b$, if either (1) $a$ and $b$ are in different ordered set partitions $\mathbf{S}_i$ and respectively $\mathbf{S}_j$, or (2) $a$ and $b$ are both in some ordered set partition $\mathbf{S}_i$ and $p_a<p_b$, respectively $p_a<p_b$.  Say that $p_a=p_b$ if and only if $a$ and $b$ are in the same block of the same ordered set partition.  Note that any composite set partition $\{\mathbf{S}_1,\ldots,\mathbf{S}_l\}$ labels a convolution product of characteristic functions of plates.  

Looking toward Theorem \ref{thm:unitriangularMap} we define preemptively the canonical bases for the four spaces $\hat{\mathcal{P}}^n$, $\mathcal{P}^n$, $\hat{\mathcal{P}}_1^n$ and $\mathcal{P}_1^n$.
	\begin{defn}
		The canonical basis for $\hat{\mathcal{P}}^n$ will consist of all characteristic functions
	$$\lbrack\lbrack\mathbf{S}_1\rbrack\rbrack\bullet\cdots\bullet\lbrack\lbrack \mathbf{S}_l\rbrack\rbrack,$$
	as $\{\mathbf{S}_1,\ldots,\mathbf{S}_l\}$ runs over all standard composite set partitions.  The canonical basis for $\mathcal{P}^n$ will be obtained from the subset of the above, consisting of all characteristic functions
	$$\lbrack\lbrack S_1,\ldots, S_k\rbrack\rbrack,$$
	as $(S_1,\ldots, S_k)$ runs over all ordered set partitions of $\{1,\ldots, n\}$ such that $1\in S_1$.  The canonical basis for $\hat{\mathcal{P}}_1^n$ will be obtained from the subset of the above, consisting of all characteristic functions
	$$\lbrack\lbrack\mathbf{S}_1\rbrack\rbrack\bullet\cdots\bullet\lbrack\lbrack \mathbf{S}_l\rbrack\rbrack$$
	labeled by standard composite ordered set partitions $(\mathbf{S}_1\rbrack\rbrack\bullet\cdots\bullet\lbrack\lbrack \mathbf{S}_l)$ of $\{1,\ldots, n\}$ where each block of each $\mathbf{S}_i$ is a singlet.  Finally, the canonical basis for $\mathcal{P}^n$ consists of the ordered set partitions labeled by permutations with first label 1:
	$$\{\lbrack\lbrack 1,i_2,\ldots, i_n\rbrack\rbrack: (i_2,\ldots, i_n)\text{ is a permutation of }\{2,\ldots, n\}\}.$$
	
	\end{defn}

	In Lemma \ref{lem:FoataSets} we define a bijection $\mathcal{U}$ between ordered set partitions $(S_1,\ldots, S_k)$ of the set $\{1,\ldots, n\}$ and standard composite set partitions.  This bijection induces an endomorphism of $\hat{\mathcal{P}}^n$ appears quite analogous to the canonical decomposition of the homogeneous component of the free Lie algebra, see the discussion around Lemma 8.22 in \cite{ReutenauerFreeLieAlgebras}.  Also note the similarity to Foata's transform.  It would be very interesting to look into these further, but we leave the investigation to future work.
	
	\begin{lem}\label{lem:FoataSets}
		There exists a bijection $\mathcal{U}$ between ordered set partitions of $\{1,\ldots, n\}$ and standard composite set partitions.
	\end{lem}
	\begin{proof}
		
		Let $(S_1,\ldots, S_k)$ be an ordered set partition of $\{1,\ldots, n\}$.  We first construct from $(S_1,\ldots, S_k)$ a standard composite set partition $\{\mathbf{S}_1,\ldots,\mathbf{S}_l\}$, where
		\begin{eqnarray*}
			\mathbf{S}_1 & = & (S_{i_l},\ldots, S_{k-1},S_{k})\\
			\mathbf{S}_2 & = & (S_{i_{l-1}},\ldots, S_{i_l-1})\\
			& \vdots &\\
			\mathbf{S}_{l-1} & = & (S_{i_2},\ldots,S_{i_3-1})\\
			\mathbf{S}_l & = & (S_1,S_2,\ldots, S_{i_2-1}).
		\end{eqnarray*}
	Let us first fix the notation $J_i=\cup_{S\in\mathbf{S}_i} S$.  The construction is inductive:
		\begin{itemize}
			\item If $1\in S_{i_l},$ define $\mathbf{S}_{1}=(S_{i_l},\ldots, S_k)$.
			\item If standard ordered set partitions $\mathbf{S}_1,\ldots, \mathbf{S}_{j-1}$ have been determined, so that $(S_1,\ldots, S_k)$ is the concatenation
			$$(S_1,\ldots, S_k)=(S_1,\ldots, S_{a},\mathbf{S}_{j-1},\ldots, \mathbf{S}_1),$$
			then define $\mathbf{S}_j$ to equal
			$$\mathbf{S}_j=(S_{i_j},\ldots, S_{a-1},S_a),$$
			where the block $S_{i_j}\in \{S_1,\ldots, S_a\}$ is characterized by the property that it contains the smallest element $s_{i_j}\in\{1,\ldots, n\}$ in the complement in $(S_1,\ldots, S_k)$ of the segment to the right of $\mathbf{S}_{j-1}$, inclusive:
			$$s_{i_j}=\text{min}\left(\{1,\ldots, n\}\setminus \left(J_{j-1}\cup\cdots\cup J_1\right)\right).$$
			
		\end{itemize}

		
		
		
		Conversely, if $\{\mathbf{S}_1,\ldots, \mathbf{S}_l\}$ is a standard composite ordered set partition with minimal first block elements respectively $s_1,\ldots, s_l$, let $(s_{i_1},\ldots, s_{i_l})$ be the permutation of $(s_1,\ldots, s_l)$ such that $s_{i_1}>\cdots>s_{i_l}$.  We then reconstruct the ordered set partition $(S_1,\ldots, S_k)$ as the concatenation $(\mathbf{S}_{i_l},\mathbf{S}_{i_{l-1}},\ldots\mathbf{S}_{i_1})$.
	\end{proof}
	
%

	In Theorem \ref{thm:unitriangularMap}, combining Theorem \ref{thm: shuffleLump Identity} and Lemma \ref{lem:FoataSets}, starting from the bijection $\mathcal{U}$ we shall induce a linear map (which we also denote by $\mathcal{U}$) to derive an automorphism of $\hat{\mathcal{P}}^n$, proving linear independence of the candidate \textit{canonical basis} 
$$\left\{\lbrack\lbrack\mathbf{S}_1\rbrack\rbrack\bullet\cdots\bullet\lbrack\lbrack \mathbf{S}_l\rbrack\rbrack:\{\mathbf{S}_1,\ldots,\mathbf{S}_l\} \text{ is a standard composite set partition of $\{1,\ldots, n\}$}\right\}$$
of $\hat{\mathcal{P}}^n$.  Namely, for the first step we have the linear map defined on the basis by 
$$\lbrack\pi\rbrack\mapsto \mathcal{U}(\lbrack\pi\rbrack)$$
where on the right-hand side $\mathcal{U}(\lbrack\pi\rbrack)=\lbrack\lbrack \mathbf{S}_1\rbrack\rbrack\bullet\cdots\bullet \lbrack\lbrack \mathbf{S}_l\rbrack\rbrack$, say, is labeled by the standard composite set partition $\{\mathbf{S}_1,\ldots, \mathbf{S}_l\}$.  As for the second step, by Theorem \ref{thm: shuffleLump Identity} this expands in the original basis of $\hat{\mathcal{P}}^n$ consisting of plates $\lbrack \pi\rbrack$:
$$\mathcal{U}(\lbrack\pi\rbrack) = \sum_\pi(-1)^{m-\operatorname{len}(\pi)}\lbrack\pi\rbrack,$$
where $m=\text{len}(\mathbf{S}_1)+\cdots+\text{len}(\mathbf{S}_l)$ and $\pi$ runs over all shuffle-lumpings of the standard ordered set partitions $\mathbf{S}_1,\ldots, \mathbf{S}_l$.

	It may be interesting to compare Theorem \ref{thm:unitriangularMap} in what follows with Theorem 5.1 of \cite{ReutenauerFreeLieAlgebras} on the construction of the Hall basis of the free Lie algebra, using the set of Lyndon words, ordered alphabetically, for the Hall set.

	We define a map $\mathcal{C}$ from the set of ordered set partitions to the set of \textit{packed words} in $\{0,\ldots, n-1\}$, that is sequences $\{c_1,\ldots, c_n\}\in\{0,\ldots, n-1\}^n$ satisfying the conditions
	\begin{enumerate}
		\item $0\in\{c_1,\ldots, c_n\}$ and 
		\item Successive values increase in steps of 1.
	\end{enumerate}
	
	\begin{prop}\label{prop: bijection}
		There exists a bijection
		$$\mathcal{C}:\{\text{Ordered set partitions of } \{1,\ldots, n\}\} \rightarrow \{\text{packed words in } \{0,\ldots, n-1\}\}.$$ 
		\end{prop}

		\begin{proof}
			Let $\mathbf{T}=(T_1,\ldots, T_{k})$ be an ordered set partition of $\{1,\ldots, n\}$.  We define a sequence $(c_1,\ldots, c_n)$ by $c_i=j-1$, if $i\in T_j$.  By construction $0=c_i$ for some $i$ and the sequence is a packed word on $\{0,\ldots, n-1\}$: if some $c\ge 1$ satisfies $c\in\{c_1,\ldots, c_n\}$ then $c-1$ satisfies the same.  Define $\mathcal{C}(\mathbf{T})=(c_1,\ldots, c_n)$.  Conversely, if $(c_1,\ldots, c_n)$ is a packed word on $\{0,\ldots, n-1\}$, define an ordered set partition $\mathbf{T}=(T_1,\ldots, T_k)$ by letting $T_i=\{m\in\{1,\ldots, n\}:c_m=i-1\}$.  
		\end{proof}
In what follows, we induce a partial order on the ordered set partitions from the lexicographic order on packed words in $\{0,\ldots, n-1\}$.
	
	\begin{defn}\label{defin:lexicographicorder}
		Given two plates $\pi_1=\lbrack(T_1,\ldots, T_k)\rbrack,\ \pi_2=\lbrack(T_1',\ldots, T_{l}')\rbrack$, say that $\pi_1 \prec \pi_2$ if 
		$$\mathcal{C}(T_1,\ldots, T_k)<\mathcal{C}(T_1',\ldots, T_l')$$ in the lexicographic order.  This induces a total order on the set of plates, and thus on the basis of their characteristic functions, in $\hat{\mathcal{P}}^n$.
	\end{defn}
	
	Clearly the first element of ordered basis is the characteristic function of $V_0^n$ itself, since it is labeled by the trivial ordered set partition, hence $\mathcal{C}(\lbrack 12\cdots n\rbrack)=(0,\ldots,0)$.  Similarly, the last element is labeled by the ordered set partition $(\{n\},\{n-1\},\ldots, \{2\},\{1\})$, where we have $\mathcal{C}(\lbrack n,n-1,\ldots, 2,1\rbrack)=(n-1,n-2\ldots,1,0)$.

We now come to our main result.
	\begin{thm}\label{thm:unitriangularMap}
		Let $\mathcal{B}^n$ be the set of characteristic functions of plates for the space $\hat{\mathcal{P}}^n$, labeled by ordered set partitions of $\{1,\ldots, n\}$.  By Proposition \ref{Prop: UniversalPlateModuleFreelyGenerated} this is a basis, which we order lexicographically.  Then, the set 
		$$\left\{\lbrack\lbrack\mathbf{S}_1\rbrack\rbrack\bullet\cdots\bullet\lbrack\lbrack \mathbf{S}_l\rbrack\rbrack:\{\mathbf{S}_1,\ldots,\mathbf{S}_l\} \text{ is a standard composite set partition of $\{1,\ldots, n\}$}\right\}$$
		is also a basis.

	\end{thm}

%
	
	\begin{proof}
			By Theorem \ref{thm: shuffleLump Identity} we have
		$$\lbrack\lbrack\mathbf{S}_1\rbrack\rbrack\bullet \cdots\bullet \lbrack\lbrack \mathbf{S}_{l}\rbrack\rbrack= \sum_\pi(-1)^{m-\operatorname{len}(\pi)}\lbrack\pi\rbrack,$$
		where $m=k_1+\cdots+k_l$ and $\pi$ runs over all shuffle-lumpings of the ordered set partitions $\mathbf{S}_1,\ldots,\mathbf{S}_l$.
		
		We prove that the matrix of the endomorphism 
		$$\lbrack\pi\rbrack\mapsto \mathcal{U}(\lbrack\pi\rbrack)=\lbrack\lbrack\mathbf{S}_1\rbrack\rbrack\bullet \cdots\bullet \lbrack\lbrack \mathbf{S}_{l}\rbrack\rbrack=\sum_\pi(-1)^{m-\operatorname{len}(\pi)}\lbrack\pi\rbrack $$
		is upper triangular with 1's on the diagonal, with respect to the lexicographically ordered basis $\mathcal{B}^n$ of plates.
	
		First note that as $\mathcal{U}(\lbrack\pi\rbrack)$ contains $\lbrack\pi\rbrack$ itself as a summand, it suffices to prove that $\mathcal{U}$ is order non-increasing with respect to the lexicographic ordering from Definition \ref{defin:lexicographicorder} on ordered set partitions.  
		
		With $\{\mathbf{S}_1,\ldots, \mathbf{S}_l\}$ the standard composite set partition coming from Lemma \ref{lem:FoataSets} applied to the plate $\pi=\lbrack S_1,\ldots, S_k\rbrack$, consider an arbitrary (signed) summand of $\mathcal{U}(\lbrack\pi\rbrack)$.  Such a summand is labeled by a shuffle-lumping $\mathbf{T}=(T_1,\ldots, T_m)$ of the ordered set partitions $\mathbf{S}_1,\ldots, \mathbf{S}_l$, with respect to the ordered set partition $(\{1\},\ldots, \{n\})$.  It follows from the construction of the ordered set partitions $\mathbf{S}_1,\ldots,\mathbf{S}_l$ from $(S_1,\ldots, S_k)$ in Lemma \ref{lem:FoataSets}, that among all of their shuffle-lumpings $(S_1,\ldots, S_k)$ itself occurs and is maximal in the lexicographic order.  It follows that the matrix for $\mathcal{U}$ is upper triangular with 1's on the diagonal.  It follows that $\mathcal{U}$ is invertible.
	\end{proof}

	One can see that elements in the canonical plate basis are convolution products $\lbrack \pi_1\rbrack \bullet \cdots\bullet \lbrack\pi_k\rbrack$ of characteristic functions of plates $\pi_1,\ldots,\pi_k$ which live in mutually orthogonal subspaces of $V_0^n$; this means that the canonical plate basis consists of Cartesian products of standard plates.
	
	\begin{example}
		With respect to the lexicographic order 
		$$(0,0,0),(0,0,1),(0,1,0),(0,1,1),(0,1,2),(0,2,1),(1,0,0),(1,0,1),$$
		$$(1,0,2),(1,1,0),(1,2,0),(2,0,1),(2,1,0),$$
		via the bijection $\mathcal{C}$, for the basis of characteristic functions of plates we have respectively
		$$\lbrack\lbrack123\rbrack\rbrack,\lbrack\lbrack12,3\rbrack\rbrack,\lbrack\lbrack13,2\rbrack\rbrack,\lbrack\lbrack1,23\rbrack\rbrack,\lbrack\lbrack1,2,3\rbrack\rbrack,\lbrack\lbrack1,3,2\rbrack\rbrack,\lbrack\lbrack23,1\rbrack\rbrack,\lbrack\lbrack2,13\rbrack\rbrack,$$
		$$\lbrack\lbrack2,1,3\rbrack\rbrack,\lbrack\lbrack3,12\rbrack\rbrack,\lbrack\lbrack3,1,2\rbrack\rbrack,\lbrack\lbrack2,3,1\rbrack\rbrack,\lbrack\lbrack3,2,1\rbrack\rbrack$$
		and the map $\mathcal{U}$ takes the form
		$$\left(
		\begin{array}{ccccccccccccc}
		1 & 0 & 0 & 0 & 0 & 0 & -1 & -1 & 0 & -1 & 0 & 0 & 1 \\
		0 & 1 & 0 & 0 & 0 & 0 & 0 & 0 & -1 & 1 & 0 & -1 & -1 \\
		0 & 0 & 1 & 0 & 0 & 0 & 0 & 1 & 0 & 0 & -1 & 0 & -1 \\
		0 & 0 & 0 & 1 & 0 & 0 & 1 & 0 & -1 & 0 & -1 & 0 & -1 \\
		0 & 0 & 0 & 0 & 1 & 0 & 0 & 0 & 1 & 0 & 1 & 1 & 1 \\
		0 & 0 & 0 & 0 & 0 & 1 & 0 & 0 & 1 & 0 & 1 & 0 & 1 \\
		0 & 0 & 0 & 0 & 0 & 0 & 1 & 0 & 0 & 0 & 0 & 0 & -1 \\
		0 & 0 & 0 & 0 & 0 & 0 & 0 & 1 & 0 & 0 & 0 & -1 & -1 \\
		0 & 0 & 0 & 0 & 0 & 0 & 0 & 0 & 1 & 0 & 0 & 1 & 1 \\
		0 & 0 & 0 & 0 & 0 & 0 & 0 & 0 & 0 & 1 & 0 & 0 & -1 \\
		0 & 0 & 0 & 0 & 0 & 0 & 0 & 0 & 0 & 0 & 1 & 0 & 1 \\
		0 & 0 & 0 & 0 & 0 & 0 & 0 & 0 & 0 & 0 & 0 & 1 & 1 \\
		0 & 0 & 0 & 0 & 0 & 0 & 0 & 0 & 0 & 0 & 0 & 0 & 1 \\
		\end{array}
		\right),
		$$
		where columns 7 through 13 label linear combinations of characteristic functions of plates which vanish outside a cone of codimension at least 1.  It is informative to verify (for example, graphically, using inclusion-exclusion, as in Figure \ref{fig:plate-relations-complete}) that the rightmost column, an alternating sum over all 13 plates, encodes the characteristic function of the point $(0,0,0)$, and that column 7 expresses 
		$$\lbrack\lbrack23,1\rbrack\rbrack\mapsto \lbrack\lbrack23,1 \rbrack\rbrack+\lbrack\lbrack1,23 \rbrack\rbrack-\lbrack\lbrack123\rbrack\rbrack = \lbrack\lbrack1 \rbrack\rbrack\bullet \lbrack\lbrack23\rbrack\rbrack.$$
		Finally, column 12 encodes
		$$\lbrack\lbrack2,3,1 \rbrack\rbrack\mapsto\lbrack\lbrack2,3,1 \rbrack\rbrack +\lbrack\lbrack2,1,3\rbrack\rbrack - \lbrack\lbrack2,13 \rbrack\rbrack+\lbrack\lbrack1,2,3\rbrack\rbrack- \lbrack\lbrack12,3\rbrack\rbrack= \lbrack\lbrack1 \rbrack\rbrack\bullet \lbrack\lbrack2,3\rbrack\rbrack,$$
		or, in an order in which it is perhaps easier to see the shuffle-lumping, 
		$$\lbrack\lbrack1 \rbrack\rbrack\bullet \lbrack\lbrack2,3\rbrack\rbrack=\lbrack\lbrack2,3,1 \rbrack\rbrack - \lbrack\lbrack2,13 \rbrack\rbrack +\lbrack\lbrack2,1,3\rbrack\rbrack- \lbrack\lbrack12,3\rbrack\rbrack+\lbrack\lbrack1,2,3\rbrack\rbrack.$$
	\end{example}
	
	The following more involved example will serve to illustrate the upper-triangularity of Theorem \ref{thm:unitriangularMap}.
	\begin{example}
		We have
		$$\mathcal{U}(\lbrack\lbrack 4\ 11,10,3,5\ 7,6\ 8,1\ 9,2\rbrack\rbrack)=\lbrack\lbrack 1\ 9,2\rbrack\rbrack \bullet\lbrack\lbrack 3,5\ 7,6\ 8\rbrack\rbrack\bullet \lbrack\lbrack 4\ 11,10\rbrack\rbrack,$$
		where we omit the (rather long) alternating sum over all shuffle-lumpings of the ordered set partitions
		$$(1\ 9,2), (3, 5\ 7,6\ 8), (4\ 11,10).$$
		It is a useful exercise to apply Lemma \ref{lem:FoataSets} to check that 
		$$\mathcal{C}(\{4,11\},\{10\},\{3\},\{5,7\},\{6,8\},\{1,9\},\{2\})=(5,6,2,0,3,4,3,4,5,1,0)$$ and verify that $(\{4,11\},\{10\},\{3\},\{5,7\},\{6,8\},\{1,9\},\{2\})$ is the shuffle-lumping that is maximal with respect to the lexicographic ordering: switching or merging any two blocks which are not in the same ordered set partition results in a lexicographically smaller ordered set partition.  For example, the shuffle-lumping obtained by merging $\{3\}$ and $\{10\}$, which are in distinct ordered set partitions, obviously decreases the lexicographic order:
		$$\mathcal{C}(\{4,11\},\{3,10\},\{5,7\},\{6,8\},\{1,9\},\{2\})=(4,5,1,0,2,3,2,3,4,1,0),$$
		as does switching $\{3\}$ and $\{10\}$:
		$$\mathcal{C}(\{4,11\},\{3\},\{10\},\{5,7\},\{6,8\},\{1,9\},\{2\})=(5,6,1,0,3,4,3,4,5,2,0).$$
		
	\end{example}

	As a consequence of Theorem \ref{thm:unitriangularMap} we have Corollary \ref{cor:Lineardimensions}.
	
	\begin{cor}\label{cor:Lineardimensions}
		The linear dimension of the degree $k$ component $(\hat{\mathcal{P}}^n)_k$ of the space $\hat{\mathcal{P}}^n$, consisting of linear combinations of characteristic functions of total dimension $k$ Minkowski sums of plates, is equal to the number of standard composite set partitions $\{\mathbf{S}_1,\ldots, \mathbf{S}_k\}$ of $\{1,\ldots, n\}$.  Namely,
		$$\dim((\hat{\mathcal{P}}^n)_k)=\sum_{i=k}^n S(n,i)s(i,k).$$
		Likewise, the linear dimension of the degree $k$ component of the space $\hat{\mathcal{P}}_1^n$, consisting of linear combinations of characteristic functions of \underline{pointed}, total dimension $k$ Minkowski sums of plates, is equal to the $k$th Stirling number of the first kind, 
		$$\dim((\hat{\mathcal{P}}^n_{1})_k)= S(n,n)s(n,k)=s(n,k).$$
		Here $S(n,i)$ is the Stirling number of the second kind, which counts the number of set partitions of $\{1,\ldots, n\}$ into $i$ blocks, and $s(i,k)$ is the Stirling number of the first kind, which counts the number of permutations of $\{1,\ldots, i\}$ which decompose as a product of $k$ disjoint cycles.
	\end{cor}
	
	\begin{proof}
		In the formula for $\dim((\hat{\mathcal{P}}^n)_k)$, the contribution $S(n,i)s(i,k)$ is the product of the number of set partitions of $n$ with $i$ blocks, times the number of permutations of $\{1,\ldots, i\}$ which decompose into $k$ disjoint cycles.  This is exactly the enumeration of the standard composite set partitions $\{\mathbf{S}_1,\ldots, \mathbf{S}_k\}$: each $\mathbf{S}_a$ is a standard ordered set partition and 
		$$\bigcup_{S\in \mathbf{S}_1} S\cup\cdots\cup \bigcup_{S\in \mathbf{S}_k} S=\{1,\ldots, n\}.$$
		
		The formula for the dimension of the $k^\text{th}$ graded component $(\hat{\mathcal{P}}_{1}^n)_k$ follows by taking the unique ordered set partition $(\{1\},\ldots, \{n\})$ of $\{1,\ldots, n\}$ and counting the number of permutations of $\{1,\ldots, n\}$ which decompose into $k$ disjoint cycles.
	\end{proof}
	
	\begin{example}
		The formula in Corollary \ref{cor:Lineardimensions} is given in O.E.I.S. A079641, the matrix product of the Stirling numbers of the second kind with the unsigned Stirling numbers of the first kind.  The first six rows are given below; note that the rows sum to the ordered Bell numbers $( 1, 3, 13, 75, 541, 4683)$.
		$$
		\begin{array}{cccccc}
		1&&&&&\\
		2 & 1 &  &  &  &\\
		6 & 6 & 1 &  &  &\\
		26 & 36 & 12 & 1 && \\
		150 & 250 & 120 & 20 & 1 & \\
		1082 & 2040 & 1230 & 300 & 30 & 1 \\
		\end{array}
		$$
		The canonical basis for $\hat{\mathcal{P}}_1^n$ has graded dimension given by the Stirling numbers of the first kind,
		$$
		\begin{array}{cccccc}
		1&&&&&\\
		1 & 1 &  &  &  &\\
		2 & 3 & 1 &  &  &\\
		6 & 11 & 6 & 1 && \\
		24 & 50 & 35 & 10 & 1 & \\
		120 & 274 & 225 & 85 & 15 & 1. \\
		\end{array}
		$$
		
	\end{example}

	\section{Plates and trees}\label{sec:PlatesTrees}

	Let $\{(i_1,j_1),\ldots, (i_{n-1},j_{n-1})\}$ be the set of oriented edges of a directed tree $\mathcal{T}$ on the vertex set $\{1,\ldots, n\}$.  This data encodes a certain permutohedral cone which is also simplicial, given explicitly as the conical hull
	$$\pi_{\mathcal{T}}=\langle e_{i_1}-e_{j_1},\ldots, e_{i_{n-1}}-e_{j_{n-1}}\rangle_+.$$
	
	In Theorem \ref{thm:treeExpansionCompletePlates} we present a combinatorial formula which expands the characteristic function of the permutohedral cone assigned to any oriented tree on $n$ vertices as a signed sum of characteristic functions of plates.  
	
	The proof of Theorem \ref{thm:treeExpansionCompletePlates} follows closely that of Theorem \ref{thm: shuffleLump Identity}.  The idea is to decompose a union of overlapping closed Weyl chambers into a disjoint union of partially open Weyl chambers; then the characteristic function of the disjoint union expands using the formula in Lemma \ref{lem: Weyl chamber decomposition shuffle}.  Then we dualize to get the permutohedral cone $\pi_{\mathcal{T}}$.

	Note that Theorem \ref{thm:treeExpansionCompletePlates} appears to first order in \cite{HeSon} as a shuffle-identity among rational functions.

	\begin{thm}\label{thm:treeExpansionCompletePlates}
		Let $\mathcal{T}=\{(i_1,j_1),\ldots(i_{n-1},j_{n-1})\}$ be a directed tree.  We have, in the space $\hat{\mathcal{P}}^n$, the identity of characteristic functions
		$$\lbrack\langle e_{i_1}-e_{j_1},\ldots, e_{i_{n-1}}-e_{j_{n-1}}\rangle_+\rbrack= \sum_{\pi:p_{i_a} < p_{j_a}}(-1)^{n-\operatorname{len}(\pi)}\lbrack\pi\rbrack,$$
		where we recall that $p_{i_a}<p_{j_a}$ if and only if in the ordered set partition which labels $\pi$, the label $i_a$ is in a block strictly to the left of the block containing $j_a$.
	\end{thm}
	
	\begin{proof}
		The dual cone is defined by the equations
		$$\langle e_{i_1}-e_{j_1},\ldots, e_{i_{n-1}}-e_{j_{n-1}}\rangle_+^\star=\left\{y\in V_0^n:y_{i_1}\ge y_{j_1},\ldots, y_{j_{n-1}}\ge y_{j_{n-1}}\right\}.$$
		We first claim that this is a union of those Weyl chambers $\cup_\tau\lbrack\tau\rbrack^\star$ defined by $y_{\tau_1}\ge\cdots\ge y_{\tau_m}$ which satisfy the $n-1$ conditions $\tau_{i_1}>\tau_{j_1},\ \ldots,\ \tau_{i_{n-1}}>\tau_{j_{n-1}}$.  To see this, let $\bar{e}_{I_1},\ldots, \bar{e}_{I_{n-1}}$ be the basis which is orthogonally dual to $e_{i_1}-e_{j_1},\ldots, e_{i_{n-1}}-e_{j_{n-1}}$, so that $\bar{e}_{I_a}\cdot (e_{i_b}-e_{j_b})=\delta_{a,b}$.  Then we have from the corresponding vector space isomorphism a bijection of cone points
		$$\langle e_{i_1}-e_{j_1},\ldots, e_{i_{n-1}}-e_{j_{n-1}}\rangle_+\rightarrow \langle e_{i_1}-e_{j_1},\ldots, e_{i_{n-1}}-e_{j_{n-1}}\rangle_+^\star=\langle\bar{e}_{I_1},\ldots, \bar{e}_{I_{n-1}}\rangle_+$$
		defined by $e_{i_a}-e_{j_a}\mapsto \bar{e}_{I_a}$, that is 
		$$\sum_{a=1}^{n-1}t_a (e_{i_a}-e_{j_a})\mapsto \sum_{a=1}^{n-1}t_a \bar{e}_{I_a}.$$
		Let $\lbrack \alpha_1\rbrack^\star,\ldots, \lbrack \alpha_m\rbrack^\star$ be the minimal set of Weyl chambers such that 
		$$\langle e_{i_1}-e_{j_1},\ldots, e_{i_{n-1}}-e_{j_{n-1}}\rangle_+^\star\subseteq\cup_{i=1}^m\lbrack\alpha_i\rbrack^\star,$$
		which means that the permutations $\alpha_1,\ldots, \alpha_{m}$ are all compatible with the orders $(i_1,j_1),\ldots,$  $(i_{n-1},j_{n-1})$.  We show that this is an equality: for each $y\in \cup_{i=1}^m\lbrack\alpha_i\rbrack^\star$, since $\bar{e}_{I_1},\ldots, \bar{e}_{I_{n-1}}$ is a basis for $V_0^n$ we have $y=\sum_{a=1}^{n-1} t_a \bar{e}_{I_a}$ for some $t_a\in\mathbb{R}$, for equality it suffices to show that $y_{i_a}-y_{j_a}=y\cdot(e_{i_a}-e_{j_a})=t_a\ge 0$ for all $a=1,\ldots, n-1$.  But having $y_{i_a}-y_{j_a}<0$ for some $a$ would imply that $\alpha_i$ is not compatible with the order $(i_1,j_1),\ldots, (i_{n-1},j_{n-1})$.
		
		As in Theorem \ref{thm: shuffleLump Identity}, we replace the Weyl chambers $\lbrack\tau\rbrack^\star$ with the (mutually disjoint) partially open Weyl chambers $C_\tau$ from Lemma \ref{lem: Weyl chamber decomposition shuffle}.  By construction these all satisfy the inequalities defining the dual cone, and we correspondingly have, for characteristic functions,
		$$\lbrack\langle e_{i_1}-e_{j_1},\ldots, e_{i_{n-1}}-e_{j_{n-1}}\rangle_+^\star\rbrack=\sum_\tau\lbrack C_\tau\rbrack,$$
		where the sum is over all permutations $\tau=(\tau_1,\ldots, \tau_n)$ satisfying the $n-1$ conditions $\tau_{i_1}>\tau_{j_1},\ \ldots,\ \tau_{i_{n-1}}>\tau_{j_{n-1}}$.

		But from Lemma \ref{lem: Weyl chamber decomposition shuffle}, for each such $\tau$  we have the further decomposition 
		$$\lbrack C_\tau\rbrack=\sum_\pi(-1)^{n-\operatorname{len}(\pi)}\lbrack\pi^\star\rbrack,$$
		where the sum is over all plates $\pi=\lbrack S_1,\ldots, S_k\rbrack$ which are labeled by ordered set partitions $(S_1,\ldots, S_k)$ such that each block is labeled by a permutation which has the set of consecutive descents of $\tau$, of the form $\tau_{i_1}>\tau_{i_2}>\cdots>\tau_{i_{\vert S_i\vert}}.$  Summing over all such $\tau$ we obtain
		$$\lbrack\langle e_{i_1}-e_{j_1},\ldots, e_{i_{n-1}}-e_{j_{n-1}}\rangle_+\rbrack= \sum_{\pi:p_{i_a} < p_{j_a}}(-1)^{n-\operatorname{len}(\pi)}\lbrack\pi\rbrack,$$
		which completes the proof.
		
	\end{proof}

	\begin{example}\label{Example: tree relation}
		Let $\mathcal{T}=\langle e_1-e_2,e_1-e_3\rangle_+$.  Then
		$$\lbrack \langle e_1-e_2,e_1-e_3\rangle_+\rbrack = \lbrack\lbrack 1,2,3\rbrack\rbrack+\lbrack\lbrack 1,3,2\rbrack\rbrack-\lbrack\lbrack 1,23\rbrack\rbrack.$$
	\end{example}
	\begin{figure}[h!]
		\centering
		\includegraphics[width=.7\linewidth]{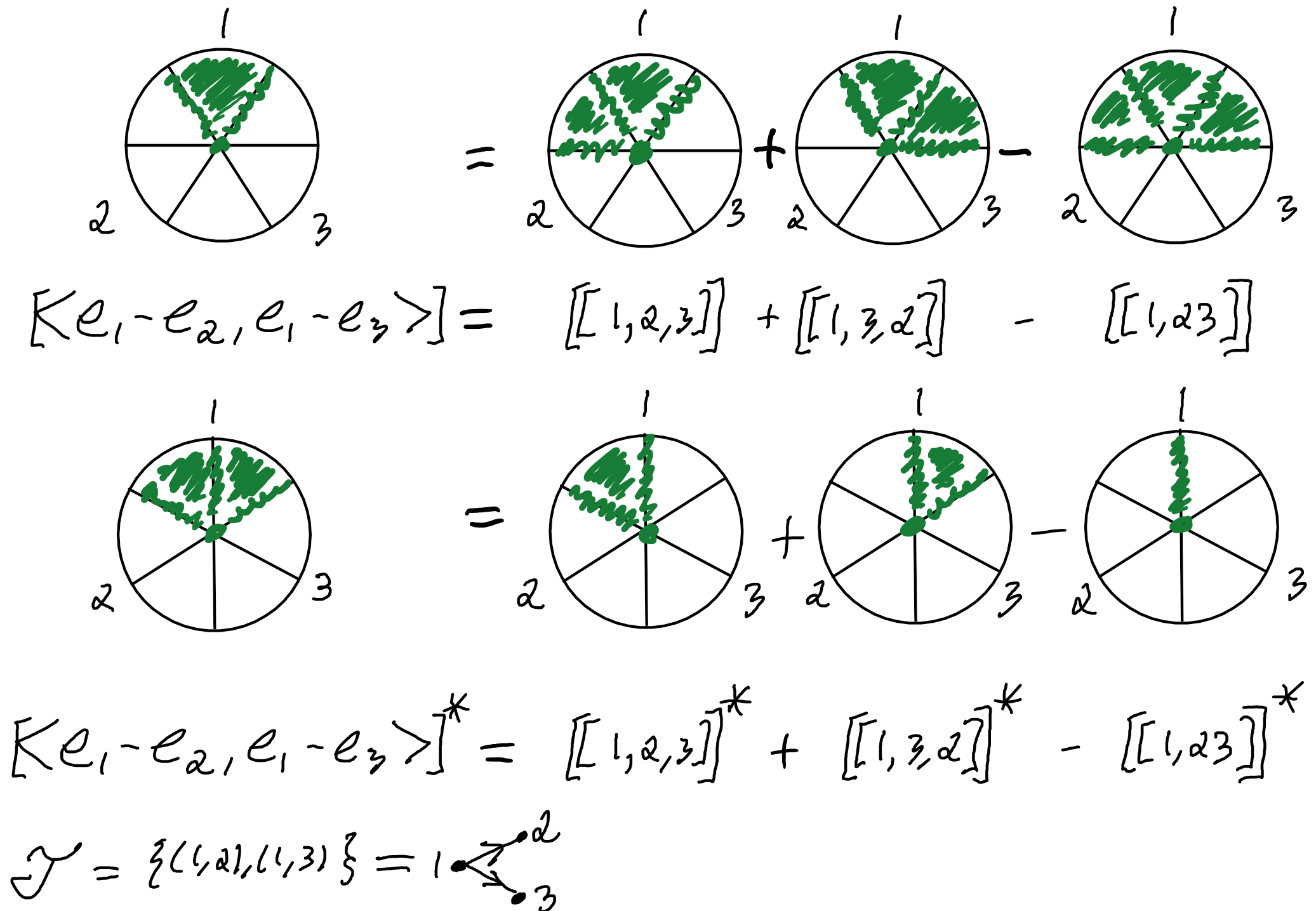}
		\caption{The relation for Example \ref{Example: tree relation} and its dual}
		\label{fig:exampletreeplaterelation}
	\end{figure}

\begin{cor}\label{cor:treeSetpartition}
	Let $\mathcal{T}= \{(i_1,j_1),\ldots, (i_{k-1},j_{k-1})\}$ be a directed tree, where $i_a,j_a\in\{1,\ldots, k\}$ with $i_a\not=j_a$.  Let $\{S_1,\ldots, S_k\}$ be a collection of disjoint nonempty subsets of $\{1,\ldots, n\}$.  We have
	$$\lbrack\lbrack S_{i_1},S_{j_1}\rbrack\rbrack\bullet\cdots\bullet \lbrack\lbrack S_{i_{k-1}},S_{j_{k-1}}\rbrack\rbrack=\sum_{\pi:p_{i_a} < p_{j_a}}(-1)^{k-\operatorname{len}(\pi)}\lbrack\pi\rbrack.$$
\end{cor}

	
	\begin{proof}
		The proof of Theorem \ref{thm:treeExpansionCompletePlates} generalizes with minimal adjustment to the present case, when $\{1,2,\ldots, n\}$ is replaced by any collection of disjoint nonempty subsets $(S_1,\ldots, S_k)$.
	\end{proof}

	\section{Straightening plates to the canonical basis}\label{sec:straightening}

	We prove in Theorem \ref{thm:Straightening} the general expression for the expansion of a plate in the canonical basis for $\hat{\mathcal{P}}^n$.  This implies the result of Ocneanu's original computation of the \textit{plate relations} in which he worked in a vector space generated formally by rooted binary trees.  Note that in practice one often works in one of the quotient spaces $\mathcal{P}^n$, $\hat{\mathcal{P}}_1^n$ or $\mathcal{P}_1^n$, see Corollary \ref{cor:plate relations quotient}.
	
	Recall that if $\{\mathbf{T}_1,\ldots, \mathbf{T}_m\}$ is a composite set partition of $\{1,\ldots, n\}$, then we say that $p_i<p_j$ if either of the following two conditions hold:
	\begin{enumerate}
		\item $i$ and $j$ are in the same ordered set partition $\mathcal{T}_a$ with $p_i<p_j$, or
		\item $i$ and $j$ are in different ordered set partitions $\mathbf{T}_a$ and $\mathbf{T}_b$, respectively, with $a\not=b$.
	\end{enumerate}
	Further, we say that $p_i=p_j$ if $i$ and $j$ are in the same block of the same ordered set partition.

\begin{thm}\label{thm:Straightening}
	Let $(S_l,S_{l-1},\ldots, S_1,S_{l+1},\ldots, S_k)$ be an ordered set partition of $\{1,\ldots, n\}$.  We have
	\begin{eqnarray*}
		&&\lbrack\lbrack S_l,S_{l-1},\ldots, S_1,S_{l+1},\ldots, S_k\rbrack\rbrack \\
		& = & \sum_{m=1}^k\sum_{\{\mathbf{T}_1,\ldots,\mathbf{T}_m\}}\sum_{\pi}(-1)^{c}(-1)^{(\text{len}(\mathbf{T}_1)+(k-l))-\text{len}(\pi)}\lbrack\lbrack\mathbf{T}_2\rbrack\rbrack \bullet\cdots\bullet \lbrack\lbrack\mathbf{T}_m\rbrack\rbrack\bullet\lbrack\pi\rbrack,
	\end{eqnarray*}
	where the sum $\sum_\pi$ varies over all ordered set partitions of 
	$$\left(\bigcup_{U\in \mathbf{T}_1} U\right) \cup S_{l+1}\cup\cdots\cup S_k$$
	satisfying the orientations 
	$$p_{(S_l)}\ge p_{(S_{l-1})}\ge \cdots\ge p_{(S_1)}<p_{(S_{l+1})}<\cdots<p_{(S_k)},$$
	using the symbol $p_{(S)}$ in place of $(p_{i_1}=\cdots =p_{i_{\vert S\vert}})$ whenever $S=\{i_1,\ldots, i_{\vert S\vert}\}\subseteq\{1,\ldots,n\}$.
	The outer sum is over all composite set partitions $\{\mathbf{T}_1,\ldots,\mathbf{T}_m\}$ of $\cup_{i=1}^l S_i$ satisfying the same orientations 
	$$p_{(S_l)}\ge p_{(S_{l-1})}\ge \cdots\ge p_{(S_1)}<p_{(S_{l+1})}<\cdots<p_{(S_k)}.$$
	  Here $c=\vert \{i\in \{1,\ldots, l-1\}: p_i<p_{i+1}\}\vert$ counts the pairs $(S_i,S_{i+1})$ such that $S_i$ and $S_{i+1}$ are in the same ordered set partition $\mathbf{T}_a$ with the (strict) orientation $p_i<p_{i+1}$.	
\end{thm}	
	
\begin{proof}
	Let us use the notation 
	\begin{itemize}
		\item $\lbrack\lbrack S,S'\rbrack\rbrack^{(c_1)}=\lbrack\lbrack S' S\rbrack\rbrack$,
		\item $\lbrack\lbrack S,S'\rbrack\rbrack^{(c_2)}=\lbrack\lbrack S', S\rbrack\rbrack$,
		\item $\lbrack\lbrack S,S'\rbrack\rbrack^{(c_3)}=\lbrack\lbrack S\rbrack\rbrack\bullet \lbrack\lbrack S'\rbrack\rbrack$.
	\end{itemize}
	Due to the relation	
$$\lbrack\lbrack S,S'\rbrack\rbrack = \lbrack\lbrack S'S\rbrack\rbrack - \lbrack\lbrack S',S\rbrack\rbrack +\lbrack\lbrack S\rbrack\rbrack\bullet \lbrack\lbrack S'\rbrack\rbrack =\lbrack\lbrack S,S'\rbrack\rbrack^{(c_1)}-\lbrack\lbrack S,S'\rbrack\rbrack^{(c_2)}+\lbrack\lbrack S,S'\rbrack\rbrack^{(c_3)} ,$$
where $S,S'$ are nontrivial disjoint subsets of $\{1,\ldots, n\}$, as well as
$$\lbrack\lbrack S_l,S_{l-1},\ldots, S_1\rbrack\rbrack = \lbrack\lbrack S_l,S_{l-1}\rbrack\rbrack\bullet \lbrack\lbrack S_{l-1},S_{l-2}\rbrack\rbrack \bullet\cdots\bullet \lbrack\lbrack S_{2},S_1\rbrack \rbrack,$$
we have
\begin{eqnarray*}
	& & \lbrack\lbrack S_l,S_{l-1},\ldots,S_1\rbrack\rbrack\\
	& = & \left(\lbrack\lbrack S_{1}S_{2}\rbrack\rbrack-\lbrack\lbrack S_{1},S_{2}\rbrack\rbrack +\lbrack\lbrack S_1\rbrack\rbrack\bullet \lbrack\lbrack S_2\rbrack\rbrack\right)\bullet\cdots\bullet \left(\lbrack\lbrack S_{l-1}S_{l}\rbrack\rbrack-\lbrack\lbrack S_{l-1},S_{l}\rbrack\rbrack +\lbrack\lbrack S_{l-1}\rbrack\rbrack\bullet \lbrack\lbrack S_{l}\rbrack\rbrack\right)\\
	& = & \sum_{w\in \{c_1,c_2,c_3\}^{l-1}}(-1)^{\#_{c_2}(w)}\lbrack\lbrack S_1,S_2\rbrack\rbrack^{(w_1)}\bullet \lbrack\lbrack S_2,S_3\rbrack\rbrack^{(w_2)}\bullet\cdots\bullet \lbrack\lbrack S_{l-1},S_l\rbrack\rbrack^{(w_{l-1})}\\
	& = & \sum_{m=1}^{k}\sum_{\{\mathbf{T}_1,\ldots,\mathbf{T}_m\}}(-1)^{c}\lbrack\lbrack \mathbf{T}_1\rbrack\rbrack\bullet\cdots\bullet \lbrack\lbrack\mathbf{T}_m\rbrack\rbrack,
\end{eqnarray*}
where the sum is over all composite ordered set partitions $\{\mathbf{T}_1,\ldots, \mathbf{T}_m\}$ of $S_1\cup\cdots\cup S_l\subseteq\{1,\ldots, n\}$, for $m\in \{1,\ldots, l\}$, satisfying 
	$$p_{(S_l)}\ge p_{(S_{l-1})}\ge \cdots\ge p_{(S_1)}.$$
	Here $\#_{c_2}(w)$ counts the number of times that $c_2$ appears in $w\in \{c_1,c_2,c_3\}^{l-1}$, and 
	$$c=\vert \{i\in \{1,\ldots, l-1\}: p_i<p_{i+1}\}\vert,$$
	hence $\#_{c_2}(w) = c$ whenever
	$$\lbrack\lbrack S_1,S_2\rbrack\rbrack^{(w_1)}\bullet \lbrack\lbrack S_2,S_3\rbrack\rbrack^{(w_2)}\bullet\cdots\bullet \lbrack\lbrack S_{l-1},S_l\rbrack\rbrack^{(w_{l-1})}= \lbrack\lbrack \mathbf{T}_1\rbrack\rbrack\bullet\cdots\bullet \lbrack\lbrack\mathbf{T}_m\rbrack\rbrack.$$
Multiplying by $\lbrack\lbrack S_1,S_{l+1},\ldots, S_k\rbrack\rbrack$ we have
$$\lbrack\lbrack S_l,S_{l-1},\ldots,S_1\rbrack\rbrack\bullet  \lbrack\lbrack S_1,S_{l+1},\ldots,S_k\rbrack\rbrack=\lbrack\lbrack S_{l},S_{l-1},\ldots, S_1,S_{l+1},\ldots, S_k\rbrack\rbrack,$$
hence 
\begin{eqnarray*}
	&& \lbrack\lbrack S_{l},S_{l-1},\ldots, S_1,S_{l+1},\ldots, S_k\rbrack\rbrack \\
	& = &  \sum_{m=1}^{k}\sum_{\{\mathbf{T}_1,\ldots,\mathbf{T}_m\}}(-1)^{c}\lbrack\lbrack \mathbf{T}_1\rbrack\rbrack\bullet\cdots\bullet \lbrack\lbrack\mathbf{T}_m\rbrack\rbrack\bullet \lbrack\lbrack S_1,S_{l+1},\ldots, S_k\rbrack\rbrack\\
	& = &  \sum_{m=1}^{k}\sum_{\{\mathbf{T}_1,\ldots,\mathbf{T}_m\}}(-1)^c\lbrack\lbrack \mathbf{T}_2\rbrack\rbrack\bullet\cdots\bullet \lbrack\lbrack\mathbf{T}_m\rbrack\rbrack\bullet \left(\lbrack\lbrack\mathbf{T}_1\rbrack\rbrack\bullet \lbrack\lbrack S_1,S_{l+1},\ldots, S_k\rbrack\rbrack \right),
\end{eqnarray*}
where in the last line we have assumed that, if we denote $\mathbf{T}_1=(U_1,\ldots, U_t)$, then $S_1\subseteq U_1$.  It remains to expand the characteristic function of the permutohedral cone specified by the (directed tree) data
$$\{(p_{(U_1)}<p_{(U_2)}),\ldots, (p_{(U_{t-1})}<p_{(U_t)}),(p_{(U_1)}<p_{(S_{l+1})}),(p_{(S_{l+1})}<p_{(S_{l+2})})\ldots,(p_{(S_{k-1})}<p_{(S_{k})})\}$$
%
in terms of the canonical plate basis.  Now by Corollary \ref{cor:treeSetpartition} we have
\begin{eqnarray*}
	\left(\lbrack\lbrack\mathbf{T}_1\rbrack\rbrack\bullet \lbrack\lbrack S_1,S_{l+1},\ldots, S_k\rbrack\rbrack\right) & = & \sum_{\pi}(-1)^{(\text{len}(\mathbf{T}_1)+(k-l+1)-1)-\text{len}(\pi)}\lbrack\pi\rbrack\\
	& = & \sum_{\pi}(-1)^{(\text{len}(\mathbf{T}_1)+(k-l))-\text{len}(\pi)}\lbrack\pi\rbrack\\
\end{eqnarray*}
where the sum is over all ordered set partitions satisfying the orientations specified by the tree; note that $S_1$ must be in the first block of each $\pi$, that is all $\pi$ which appear in the sum are \textit{standard} ordered set partitions.  It follows that
$$\lbrack\lbrack S_{l},S_{l-1},\ldots, S_1,S_{l+1},\ldots, S_k\rbrack\rbrack$$
$$=\sum_{m=1}^{k}\sum_{\{\mathbf{T}_1,\ldots,\mathbf{T}_m\}}(-1)^c\lbrack\lbrack \mathbf{T}_2\rbrack\rbrack\bullet\cdots\bullet \lbrack\lbrack\mathbf{T}_m\rbrack\rbrack\bullet\left(\sum_{\pi}(-1)^{(\text{len}(\mathbf{T}_1)+(k-l))-\text{len}(\pi)}\lbrack\pi\rbrack\right)$$
$$=\sum_{m=1}^k\sum_{\{\mathbf{T}_1,\ldots,\mathbf{T}_m\}}\sum_{\pi}(-1)^{c}(-1)^{(\text{len}(\mathbf{T}_1)+(k-l))-\text{len}(\pi)}\lbrack\lbrack\mathbf{T}_2\rbrack\rbrack \bullet\cdots\bullet \lbrack\lbrack\mathbf{T}_m\rbrack\rbrack\bullet\lbrack\pi\rbrack.$$
\end{proof}

\begin{example}
	In Figure \ref{fig:exampleplaterelationsnice3} we illustrate an expansion of the characteristic function $\lbrack\lbrack 2,1,3\rbrack\rbrack$ in the canonical plate basis.
\end{example}
\begin{figure}[h!]
	\centering
	\includegraphics[width=.9\linewidth]{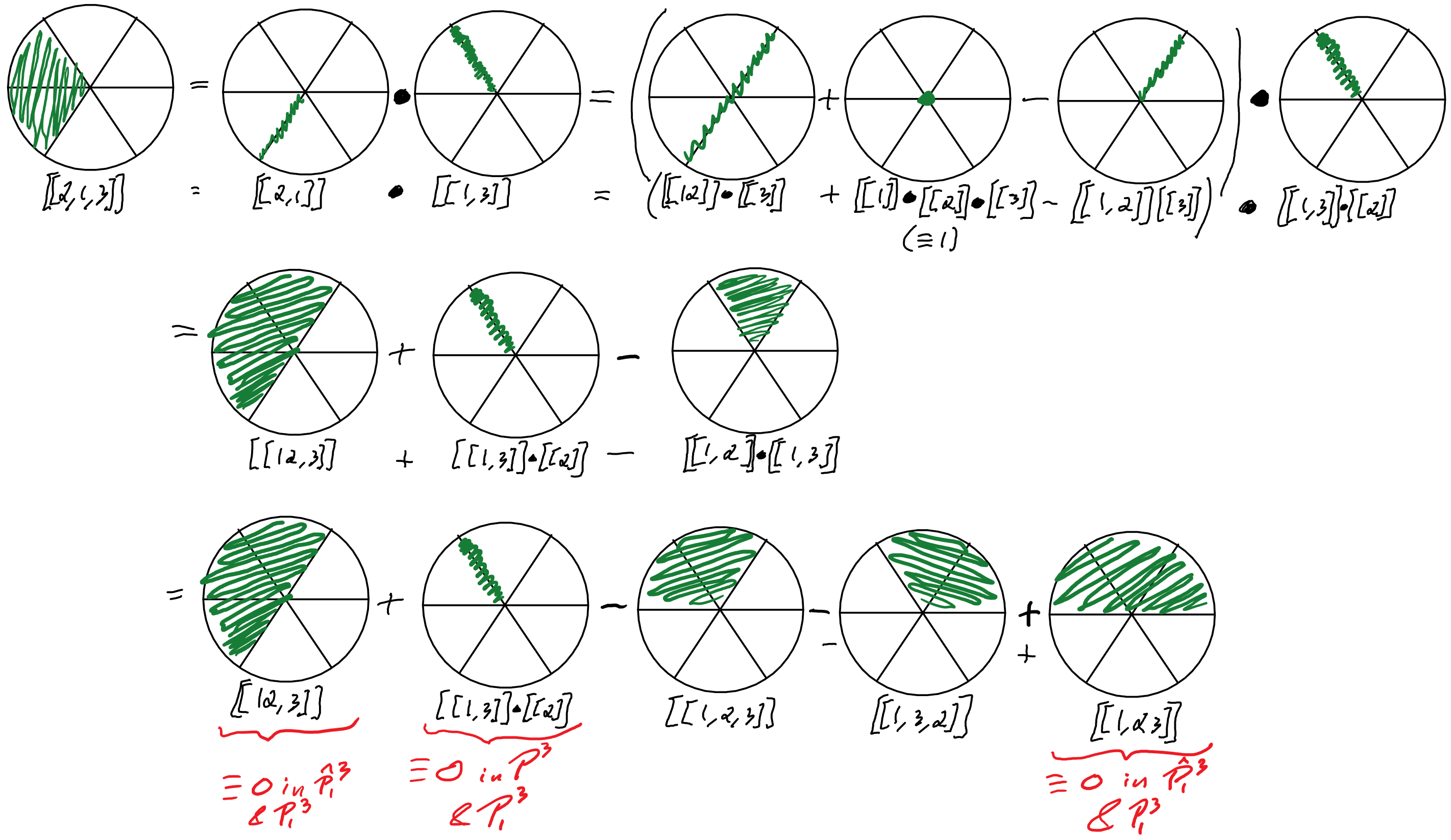}
	\caption{Some computations in the convolution algebra of permutohedral cones: straightening to the canonical plate basis}
	\label{fig:exampleplaterelationsnice3}
\end{figure}
\begin{example}
	Let us indicate the key steps in straightening $\lbrack\lbrack 3,2,1,4,5\rbrack\rbrack$ to the canonical basis.  We have
\begin{eqnarray*}
 & & \lbrack\lbrack 3,2,1,4,5\rbrack\rbrack\\
 & = &  \left(\lbrack\lbrack 23\rbrack\rbrack -\lbrack\lbrack 2,3\rbrack\rbrack +1 \right)\bullet \left(\lbrack\lbrack 12\rbrack\rbrack -\lbrack\lbrack 1,2\rbrack\rbrack +1 \right)\bullet\lbrack\lbrack 1,4,5\rbrack\rbrack\\
& = & \left(\lbrack\lbrack 123\rbrack\rbrack - \lbrack\lbrack 1,23\rbrack\rbrack +\lbrack\lbrack 23\rbrack\rbrack - \lbrack\lbrack 12,3\rbrack\rbrack + \lbrack\lbrack 1,2,3\rbrack\rbrack -\lbrack\lbrack 2,3\rbrack\rbrack + \lbrack\lbrack 12\rbrack\rbrack - \lbrack\lbrack 1,2\rbrack\rbrack+1\right)\bullet \lbrack\lbrack 1,4,5\rbrack\rbrack.
\end{eqnarray*}
These separate into two groups, one of which is already in the canonical basis and one which requires further work.  First, we have
\begin{eqnarray*}
\lbrack\lbrack 23\rbrack\rbrack\bullet\lbrack\lbrack 1,4,5\rbrack\rbrack - \lbrack\lbrack 2,3\rbrack\rbrack\bullet\lbrack\lbrack 1,4,5\rbrack\rbrack+\lbrack\lbrack 1,4,5\rbrack\rbrack.
\end{eqnarray*}
These are already in the canonical basis.

More work is required for the remainder,
\begin{eqnarray*}
\left(\lbrack\lbrack 123\rbrack\rbrack - \lbrack\lbrack 1,23\rbrack\rbrack - \lbrack\lbrack 12,3\rbrack\rbrack + \lbrack\lbrack 1,2,3\rbrack\rbrack  + \lbrack\lbrack 12\rbrack\rbrack - \lbrack\lbrack 1,2\rbrack\rbrack\right)\bullet \lbrack\lbrack 1,4,5\rbrack\rbrack.
\end{eqnarray*}
Let us consider some term-by-term computations.  We have
\begin{eqnarray*}
\lbrack\lbrack 123\rbrack\rbrack\bullet \lbrack\lbrack 1,4,5\rbrack\rbrack & = & \lbrack\lbrack 123,4,5\rbrack\rbrack\\
\lbrack\lbrack 1,23\rbrack\rbrack\bullet\lbrack\lbrack 1,4,5\rbrack\rbrack & = & \lbrack\lbrack 1,23,4,5\rbrack\rbrack - \lbrack\lbrack 1,234,5\rbrack\rbrack +\lbrack\lbrack 1,4,23,5\rbrack\rbrack -\lbrack\lbrack 1,4,235\rbrack\rbrack +\lbrack\lbrack 1,4,5,23\rbrack\rbrack \\
\lbrack\lbrack 1,2\rbrack\rbrack \bullet \lbrack\lbrack 1,4,5\rbrack\rbrack & = & \lbrack\lbrack 1,2,4,5\rbrack\rbrack - \lbrack\lbrack 1,24,5\rbrack\rbrack + \lbrack\lbrack 1,4,2,5\rbrack\rbrack-\lbrack\lbrack 1,4,25\rbrack\rbrack +\lbrack\lbrack 1,4,5,2\rbrack\rbrack,
\end{eqnarray*}
where we remark that in the last line the singlet $\lbrack\lbrack 3\rbrack\rbrack$ has been removed from the composite ordered set partitions.  It acts trivially via the convolution product, being the characteristic function of the point $(0,\ldots, 0)\in V_0^n$.  In full detail, the last line would be 
$$ -\lbrack\lbrack 1,2,4,5\rbrack\rbrack\bullet\lbrack\lbrack 3\rbrack\rbrack + \lbrack\lbrack 1,24,5\rbrack\rbrack\bullet\lbrack\lbrack 3\rbrack\rbrack - \lbrack\lbrack 1,4,2,5\rbrack\rbrack\bullet\lbrack\lbrack 3\rbrack\rbrack\\
+  \lbrack\lbrack 1,4,25\rbrack\rbrack\bullet\lbrack\lbrack 3\rbrack\rbrack-\lbrack\lbrack 1,4,5,2\rbrack\rbrack\bullet\lbrack\lbrack 3\rbrack\rbrack.$$
Note that in each term which appears in present expansion of $\lbrack\lbrack 3,2,1,4,5\rbrack\rbrack$ in the canonical basis we have $p_3\ge p_2\ge p_1<p_4<p_5$.  Note that that some orientations may hold trivially, for example in $\lbrack\lbrack 1,4,5,2\rbrack\rbrack\bullet\lbrack\lbrack 3\rbrack\rbrack,$ the orientation $ p_3\ge p_2$ is trivially satisfied.
\end{example}

Unfortunately the formula given in Theorem \ref{thm:Straightening} is somewhat inconvenient; however, in the quotient spaces $\mathcal{P}^n$, $\hat{\mathcal{P}}_1^n$ and $\mathcal{P}_1^n$ the expressions are much more manageable.

\begin{cor}\label{cor:straightening1}
	Let $\lbrack\lbrack S_l,S_{l-1},\ldots, S_1,S_{l+1},\ldots, S_k\rbrack\rbrack\in\hat{\mathcal{P}}^n$, where $(S_l,S_{l-1},\ldots,S_1,S_{l+1},\ldots,  S_k)$ is an ordered set partition of $\{1,\ldots, n\}$.  Then, passing to the quotient $\mathcal{P}^n$ we have
	\begin{eqnarray*}
		\lbrack\lbrack S_l,S_{l-1},\ldots, S_1,S_{l+1},\ldots, S_k\rbrack\rbrack
		& = & \sum_{\pi}(-1)^{k-l-1-\operatorname{len}(\pi)}\lbrack \pi\rbrack,
	\end{eqnarray*}
	where the sum is over all plates $\pi$ labeled by ordered set partitions satisfying the orientations
	$$p_{(S_l)}\ge p_{(S_{l-1})}\ge \cdots\ge p_{(S_1)}<p_{(S_{l+1})}<\cdots<p_{(S_k)}.$$
\end{cor}

\begin{proof}
	In $\mathcal{P}^n$, in the formula in Theorem \ref{thm:Straightening} only composite ordered set partitions $\{\mathbf{T}_1,\ldots, \mathbf{T}_m\}$ having $m=1$ contribute, matching the expression above except for the sign.  
	
	From Theorem \ref{thm:Straightening} the sign is $(-1)^{c}(-1)^{(\text{len}(\mathbf{T}_1)+(k-l))-\text{len}(\pi)}$.  We claim that for the quotient $\mathcal{P}^n$, the number $c+\text{len}(\mathbf{T}_1)$ is always odd.  To see this, recall that in $\mathcal{P}^n$, we have 
	$$\lbrack\lbrack S_2,S_1\rbrack\rbrack = \lbrack\lbrack S_1S_2\rbrack\rbrack - \lbrack\lbrack S_1,S_2\rbrack\rbrack,$$ 
	hence 
	$$\lbrack\lbrack S_l,S_{l-1}\rbrack\rbrack\bullet \cdots\bullet \lbrack\lbrack S_2,S_1\rbrack\rbrack =  (\lbrack\lbrack S_{1}S_2\rbrack\rbrack - \lbrack\lbrack S_1,S_2\rbrack\rbrack)\bullet\cdots\bullet (\lbrack\lbrack S_{l-1}S_{l}\rbrack\rbrack - \lbrack\lbrack S_{l-1},S_{l}\rbrack\rbrack).$$ 
	Let $\mathbf{T}_1$ be any term in the expansion; now $c$ counts the number of times the two-block terms $\lbrack\lbrack S_i,S_{i+1}\rbrack\rbrack$ appear in a product.  But it is easy to see that a convolution product of $c$ of the two-block factors and $c'$ of the one-block factors is a plate $\lbrack \mathbf{T}_1\rbrack$ with $\text{len}(\mathbf{T}_1)=c+1$ blocks, hence $c+\text{len}(\mathbf{T}_1) = c+(c+1) = 2c+1$, and so
	$$(-1)^{c}(-1)^{(\text{len}(\mathbf{T}_1)+(k-l))-\text{len}(\pi)}=(-1)^{k-l-1-\text{len}(\pi)}.$$
\end{proof}

\begin{example}
	Consider the characteristic function $\lbrack\lbrack 2,1,3\rbrack\rbrack$.  Here $l=2$ and $k=3$.
	\begin{eqnarray*}
	\lbrack\lbrack 2,1,3\rbrack\rbrack & = & \left((-1)^{c_1}\lbrack\lbrack 12\rbrack\rbrack +(-1)^{c_2}\lbrack\lbrack 1,2\rbrack\rbrack \right)\bullet \lbrack\lbrack 1,3\rbrack\rbrack\\
	& = & (-1)^{c_1}\lbrack\lbrack 12\rbrack\rbrack \bullet  \lbrack\lbrack 1,3\rbrack\rbrack + (-1)^{c_2}\lbrack\lbrack 1,2\rbrack\rbrack \bullet  \lbrack\lbrack 1,3\rbrack\rbrack\\
	& = & (-1)^{c_1}(-1)^{d_1}\lbrack\lbrack 12,3\rbrack\rbrack + (-1)^{c_2}((-1)^{d_2}\lbrack\lbrack 1,2,3\rbrack\rbrack + (-1)^{d_3}\lbrack\lbrack 1,3,2\rbrack\rbrack + (-1)^{d_4}\lbrack\lbrack 1,23\rbrack\rbrack),
	\end{eqnarray*}
where $c_1=0$ and $c_2=1$ count the respective occurrences of $p_1<p_2$.  Further, each $d_i$ is of the form
$$d_i\sim(\text{len}(\mathbf{T}_1)+(k-l))-\text{len}(\pi).$$
That is,
$$(-1)^{d_1} = (-1)^{(1+(3-2))-2} = 1,$$
$$(-1)^{d_2} = (-1)^{(2 + (3-2))-3}=1,$$
$$(-1)^{d_3} = (-1)^{(2 + (3-2))-3}=1,$$
$$(-1)^{d_4} = (-1)^{(2 + (3-2))-2}=-1.$$
For example, for the $(-1)^{d_4}$ coefficient we have $\lbrack \mathbf{T}_1\rbrack = \lbrack 1,2\rbrack$ and $\pi = \lbrack 1,23\rbrack$.

Wrapping things up, in $\mathcal{P}^3$ we have
$$\lbrack\lbrack 2,1,3\rbrack\rbrack = \lbrack\lbrack 12,3\rbrack\rbrack +\lbrack\lbrack 1,23\rbrack\rbrack - \lbrack\lbrack 1,2,3\rbrack\rbrack-\lbrack\lbrack 1,3,2\rbrack\rbrack.$$
\end{example}

Straightening relations for the remaining quotient spaces are given in Corollary \ref{cor:plate relations quotient}.

\begin{cor}\label{cor:plate relations quotient}
	Let $(i_l,i_{l-1},\ldots, i_1,i_{l+1},\ldots, i_n)$ be an ordered set partition of $\{1,\ldots, n\}$ where all blocks\ are singlets.
	
	Then in $\hat{\mathcal{P}}_1^n$ we have
	$$\lbrack\lbrack i_l,i_{l-1},\ldots,i_1,i_{l+1},\ldots, i_n\rbrack\rbrack=
	\sum_{m=1}^{n}\sum_{\{\mathbf{T}_1,\ldots,\mathbf{T}_m\}}(-1)^{l-m}\lbrack\lbrack \mathbf{T}_1\rbrack\rbrack\bullet\cdots\bullet \lbrack\lbrack\mathbf{T}_m\rbrack\rbrack,$$
	where the sum is over all (standard) composite set partitions satisfying the orientations
	$$p_{i_l}> p_{i_{l-1}}> \cdots> p_{i_1}<p_{i_{l+1}}<\cdots<p_{i_n}.$$
	Note that all blocks in the composite set partitions above are singlets.
	
	In $\mathcal{P}_1^n$ we have
	\begin{eqnarray*}
		\lbrack\lbrack i_l,i_{l-1},\ldots,i_1,i_{l+1},\ldots, i_n\rbrack\rbrack
		& = & (-1)^{l-1}\sum_{\pi}\lbrack\pi\rbrack,
	\end{eqnarray*}	
	where the sum is over all (standard) ordered set partitions satisfying the orientations
	$$p_{i_l}> p_{i_{l-1}}> \cdots> p_{i_1}<p_{i_{l+1}}<\cdots<p_{i_n}.$$
	Note that the blocks in the ordered set partitions are in this case singlets.
\end{cor}

\begin{proof}
	The proof is similar to that for Corollary \ref{cor:straightening1} and is omitted.
\end{proof}

	\section{Acknowledgements}
	We gratefully acknowledge the hospitality of the Munich Institute for Astro- and Particle Physics (MIAPP) during the program on Mathematics and Physics of Scattering Amplitudes in August, 2017, as well the Institute for Advanced Study, where parts of this paper were written.
	
	We thank Adrian Ocneanu for the many intensive discussions during our graduate work.  We thank Nima Arkani-Hamed, Freddy Cachazo, Lance Dixon, Song He, Florent Hivert, Chrysostomos Kalousios, Carlos Mafra, Jean-Christophe Novelli, Alex Postnikov and Oliver Schlotterer for very interesting related discussions at various stages of the development of the paper.  We are grateful to Darij Grinberg and Victor Reiner for proof-reading and helpful conversations.

	\newpage
	
	\appendix

	\section{Functional representations}\label{section:ExamplesPhysics}
	\subsection{Laplace transforms}
	
	Below we observe that the formula of Theorem \ref{thm:treeExpansionCompletePlates} carries over to rational function representations, in a similar way to the discussion in \cite{LinearExtensionsValuation}, thus giving several extremely useful functional representations of the spaces $\hat{\mathcal{P}}^n, \mathcal{P}^n,\hat{\mathcal{P}}_1^n$, and $\mathcal{P}_1^n$.  
	
	Following the well-known construction, see for example \cite{ReinerValuationCone}, any permutohedral cone
	$$C=\langle v_1,\ldots, v_k\rangle_+ \subset V_0^n$$ which, after scaling the $v_i$ by a common factor is contained in an integer lattice $L$, maps, via the \textit{discrete Laplace transform}, to the rational function
	$$\mathcal{L}_d(C)=\sum_{v\in C\cap L} e^{-\langle y,v\rangle}.$$
	
	We shall also need the \textit{integral Laplace transform}, which we define by
	$$\mathcal{L}_\mathcal{I}(C)=\int_C e^{-\langle y,v\rangle }dv.$$

	\begin{defn}\label{defn:functionalRepresentation}
		Let $(S_1,\ldots, S_k)$ be an ordered set partition.  Let $\mathcal{W}^n\subset V_0^n$ be the lattice of weights of type $A_{n-1}$.  We define the functional representation of \textit{type} $\mathcal{P}^n$ of a plate $\pi$ to be the composition of the duality operator $\star$ with the valuation (as in \cite{BarvinokPammersheim}) induced on characteristic functions by the well-known discrete Laplace transform:
		$$\lbrack \lbrack S_1,\ldots, S_k\rbrack\rbrack\mapsto \mathcal{L}_d\left(\lbrack \lbrack S_1,\ldots, S_k\rbrack\rbrack^\star\right)=\mathcal{L}_d\left(\lbrack\langle \bar{e}_{S_1},\bar{e}_{S_1\cup S_2},\ldots, \bar{e}_{S_1\cup\cdots\cup S_{k-1}}\rangle\rbrack \right)= \prod_{i=1}^{k-1}\frac{1}{1-e^{-\sum_{j=1}^i\bar{y}_{S_j}}},$$
		for each plate $\lbrack S_1,\ldots, S_k\rbrack\subseteq V_0^n$, where 
		$$\bar{y}_{S}=y_S-\left(\frac{\vert S\vert}{n}\sum_{i=1}^n y_i\right)=\frac{\vert S^c\vert}{n}y_S-\frac{\vert S\vert}{n}y_{S^c}.$$
		
		For any plate $\pi=\lbrack i_1,\ldots, i_n\rbrack$ whose blocks are all singlets, we define the functional representation of \textit{type} $\hat{\mathcal{P}}_1^n$ of a plate $\pi=\lbrack i_1,\ldots, i_n\rbrack$ as
		$$\lbrack\lbrack i_1,\ldots, i_n\rbrack\rbrack\mapsto \mathcal{L}_d\left(\lbrack\lbrack i_1,\ldots, i_n\rbrack\rbrack\right)=\prod_{i=1}^{n-1}\frac{1}{1-e^{-(y_i- y_{i+1})}}=(-1)^{n-1}\frac{x_2\cdots x_n}{(x_1-x_2)\cdots (x_{n-1}-x_n)},$$
		where we define $x_i=e^{-y_i}$.
		
		Finally, for any plate $\pi=\lbrack i_1,\ldots, i_n\rbrack$ whose blocks are again all singlets, we define the functional representation of \textit{type} $\mathcal{P}_1^n$ of a plate $\pi=\lbrack i_1,\ldots, i_n\rbrack$ as
		$$\lbrack\lbrack i_1,\ldots, i_n\rbrack\rbrack\mapsto \mathcal{L}_\mathcal{I}\left(\lbrack\lbrack i_1,\ldots, i_n\rbrack\rbrack\right)=\prod_{i=1}^{n-1}\frac{1}{y_i-y_{i+1}}.$$
		
	\end{defn}
	
	For the general theory of functional representations of polyhedra cones relevant to Proposition \ref{prop:functionalRepsWellDefined}, we refer to the reader to \cite{Barvinok,BarvinokPammersheim,ReinerValuationCone}.
	\begin{prop}\label{prop:functionalRepsWellDefined}
		The functional representations from  Definition \ref{defn:functionalRepresentation}, of types $\mathcal{P}^n$, $\hat{\mathcal{P}}_1^n$ and $\mathcal{P}^n$, are well-defined.
	\end{prop}
	\begin{proof}
		
		According to \cite{Barvinok,BarvinokPammersheim}, see also Proposition 2.2 and respectively Proposition 2.1 in \cite{ReinerValuationCone}, under the linear map induced by the discrete Laplace transform from characteristic functions of polyhedral cones  to rational functions, characteristic functions of non-pointed cones are mapped to zero, and in the same way the linear map induced by the integral Laplace transform sends characteristic functions of both non-pointed cones \textit{and} higher codimension cones to the zero rational function.  
		
		Recall that the functional representation of type $\mathcal{P}^n$ is the composition of $\star$ with the discrete Laplace transform.  That is, $\mathcal{L}_d(\lbrack C^\star\rbrack )=0$ if and only if $C^\star $ is not pointed.  Now $C^\star$ is not pointed if and only if $C$ is of higher codimension, by Remark \ref{rem: properties dual cones}.  But $\mathcal{P}^n$ was defined to be the quotient of $\hat{\mathcal{P}}^n$ by the linear span of the characteristic functions higher codimensional cones, showing that the functional representation of type $\mathcal{P}^n$ is indeed a (well-defined) linear homomorphism.
		
		Recall that $\hat{\mathcal{P}}_1^n$ was defined to be the quotient of $\hat{\mathcal{P}}^n$ by the linear span of the characteristic functions of all plates which have some block of size at least 2.  Suppose $\pi=\lbrack S_1,\ldots, S_k\rbrack$ and we have $\vert S_i\vert\ge 2$ for some $i$.  Then $\pi$ contains both directions $e_a-e_b$ and $e_b-e_a$ for $a,b\in S_i$ with $a\not=b$, meaning that $\mathcal{L}_d(\lbrack\pi\rbrack)=0$.  This completes the proof that the functional representation of type $\hat{\mathcal{P}}_1^n$ is well-defined.
		
		Finally, well-definedness for the functional representation of type $\mathcal{P}_1^n$ follows again by simply noting that $\mathcal{P}_1^n$ is the quotient of $\hat{\mathcal{P}}^n$ with spanning set those characteristic functions of plates which are labeled by ordered set partitions having a block of size at least 2, or which are of higher codimension.
	\end{proof}

	\begin{rem}
		Sometimes it is convenient for computational insight to change variables in our representation of type $\mathcal{P}^n$.  Put $x_i=e^{-\bar{y}_i}$; note that $x_1\cdots x_n=e^0=1$.  Then the functional representation of type $\mathcal{P}^n$ looks like
		\begin{eqnarray}\label{eqn:functional rep}
			\lbrack \lbrack S_1,\ldots, S_k\rbrack\rbrack\mapsto \prod_{i=1}^{k-1}\frac{1}{1-\prod_{j\in S_1\cup\cdots\cup S_i}x_{j}},
		\end{eqnarray}
		and in computations one is forced to consider the relation $x_1\cdots x_n=1$ on a case-by-case basis.  Then the series expansion of \eqref{eqn:functional rep} converges on the plate $\lbrack S_1,\ldots, S_k\rbrack$, in the $\bar{y}_i$ variables.
	\end{rem}

	\subsection{Examples: Laplace transforms of permutohedral cones for tree graphs}

We now illustrate the signed sum expansions of the characteristic function of the permutohedral cone associated to a tree, for the functional representations of types $\mathcal{P}^n, \hat{\mathcal{P}}_1^n$ and $\mathcal{P}_1^n$.

A tree simplex may be cut out using inequalities by taking the orthogonal dual in $V_0^n$ to the root basis $\{e_{i_a}-e_{j_a}:a=1,\ldots, n-1\}$, as specified by the tree.

\begin{prop}\label{prop: LaplaceTransformTree}
	Let $\mathcal{T}$ be a tree with oriented edge set 
	$$\{(i_1,j_1),\ldots, (i_{n-1},j_{n-1})\}.$$
	This means that $B_\mathcal{T}=\{ e_{i_1}-e_{j_1},\ldots, e_{i_{n-1}}-e_{j_{n-1}}\}$ is a basis for $V_0^n$ and the conical hull 
	$$\langle e_{i_1}-e_{j_1},\ldots, e_{i_{n-1}}-e_{j_{n-1}}\rangle_+$$
	is a simplicial cone in $V_0^n$.  Let $\{\bar{e}_{I_1},\ldots, \bar{e}_{I_{n-1}}\}$ be the basis of $V_0^n$ which is orthogonal to $B_{\mathcal{T}}$, for subsets $I_k\subset\{1,\ldots, n\}$.  Then, we have
	$$\langle e_{i_1}-e_{j_1},\ldots, e_{i_{n-1}}-e_{j_{n-1}}\rangle_+=\left\{x\in V_0^n:\bar{e}_{I_1}\cdot x\ge 0,\ldots, \bar{e}_{I_{n-1}}\cdot x\ge 0\right\}.$$
\end{prop}

\begin{proof}
	This essentially follows from Theorem \ref{thm:treeExpansionCompletePlates}; we fill in the details.
	
	Let $\{\bar{e}_{I_1},\ldots, \bar{e}_{I_{n-1}}\}$ be the basis of $V_0^n$ which is orthogonal to $ \{e_{i_1}-e_{j_1},\ldots, e_{i_{n-1}}-e_{j_{n-1}}\}$.
	Then, given $x=\sum_{k=1}^{n-1}t_k(e_{i_k}-e_{j_k})$ with $t_k\ge 0$, we have $\bar{e}_{I_k}\cdot x = t_k \ge 0$.  Conversely, any $x\in V_0^n$ has a unique expansion
	$$x = \sum_{k=1}^{n-1} t_k(e_{i_k}-e_{j_k})$$
	for some $t_1,\ldots, t_{n-1}\in V_0^n$.  If $\bar{e}_{I_k}\cdot x\ge 0$ for each subset $I_k$, $k=1,\ldots, n-1$, define $t_k=\bar{e}_{I_k}\cdot x$, and we have
	$$x=\sum_{i=1}^{n-1}t_k \bar{e}_{I_k}.$$
\end{proof}

We have the following explicit closed form for the functional representation of type $\mathcal{P}_1^n$.  It may be interesting to compare Proposition \ref{prop:functional Tree Identities} with \cite{ReinerValuationCone}, for example Corollary 5.3.
\begin{prop}\label{prop:functional Tree Identities}
	Let $\mathcal{T}$ be a tree with oriented edge set 
$$\{(i_1,j_1),\ldots, (i_{n-1},j_{n-1})\}.$$
	We have the following identities.
	\begin{enumerate}
		\item $$\sum_{\{\mathbf{T}: p_{i_k}<p_{j_k}\}} (-1)^{n-\text{len}(\mathbf{T})}\prod_{i=1}^{k-1}\frac{1}{1-e^{-\sum_{j=1}^i\bar{y}_{T_j}}} = \prod_{i=1}^{k-1}\frac{1}{1-e^{-\sum_{j=1}^i\bar{y}_{I_j}}} $$
	where the sum is over all ordered set partitions of $\{1,\ldots, n\}$ which respect the orientations $p_{i_k}<p_{j_k}$, and the sets $I_1,\ldots, I_{n-1}$ are defined such that $\{\bar{e}_{I_1},\ldots, \bar{e}_{I_{n-1}}\}$ is orthogonal to the basis $ \{e_{i_1}-e_{j_1},\ldots,e_{i_{n-1}}-e_{j_{n-1}}\}$.
	
	\item $$\sum_{\{\sigma\in\symm_n:p_{i_k}<p_{j_k}\}}\frac{1}{(1-x_{\sigma_1}/x_{\sigma_2})\cdots (1-x_{\sigma_{n-1}}/x_{\sigma_n})} = \prod_{k=1}^{n-1} \frac{1}{1-x_{i_k}/x_{j_k}},$$
	where the sum is over all permutations of $(1,\ldots, n)$ which respect the orientations $p_{i_k}<p_{j_k}$.
	\item 	$$\sum_{\{\sigma\in\symm_n:p_{i_k}<p_{j_k}\}}\frac{1}{(x_{\sigma_1}-x_{\sigma_2})\cdots (x_{\sigma_{n-1}}-x_{\sigma_n})} = \prod_{k=1}^{n-1} \frac{1}{x_{i_k}-x_{j_k}},$$
	where again the sum is over all permutations of $(1,\ldots, n)$ which respect the orientations $p_{i_k}<p_{j_k}$.
\end{enumerate}

\end{prop}

\begin{proof}
	Recall the change of variable $x_i = e^{-y_i}$.
	
	From Theorem \ref{thm:treeExpansionCompletePlates} in $\hat{\mathcal{P}}^n$ we have the identity
	$$\lbrack\langle e_{i_1}-e_{j_1},\ldots, e_{i_{n-1}}-e_{j_{n-1}}\rangle_+\rbrack= \sum_{\pi:p_{i_a} < p_{j_a}}(-1)^{n-\operatorname{len}(\pi)}\lbrack\pi\rbrack,$$
	where the sum is over all ordered set partitions compatible with the orientations 
	$$p_{i_1}<p_{j_1},\ \ldots,\ p_{i_{n-1}}<p_{j_{n-1}}.$$
	
	Now we apply the formulas provided in the definitions of the functional representations, noting that from Proposition \ref{prop:functionalRepsWellDefined}, the functional representations of types $\mathcal{P}^n$, $\hat{\mathcal{P}}_1^n$ and $\mathcal{P}_1^n$ respect the identity from Theorem \ref{thm:treeExpansionCompletePlates}.  Therefore it suffices to compute the images under the functional representations of 
	$$\lbrack\langle e_{i_1}-e_{j_1},\ldots, e_{i_{n-1}}-e_{j_{n-1}}\rangle_+\rbrack.$$

	In the functional representation of type $\mathcal{P}^n$ we have

\begin{eqnarray*}
	\lbrack \langle e_{i_1}-e_{j_1},\ldots,e_{i_{n-1}}-e_{j_{n-1}}\rangle \rbrack  & \mapsto & \mathcal{L}_d\left(\lbrack \langle e_{i_1}-e_{j_1},\ldots,e_{i_{n-1}}-e_{j_{n-1}}\rangle ^\star\right)\\
	& = & \mathcal{L}_d\left(\lbrack\langle \bar{e}_{I_1},\bar{e}_{I_2},\ldots, \bar{e}_{I_{n-1}}\rangle\rbrack \right)\\
	& = & \sum_{k_1,\ldots, k_{n-1}=0}^\infty e^{-(k_1\bar{y}_{I_1})}e^{-(k_2\bar{y}_{I_2})}\cdots e^{-(k_{n-1}\bar{y}_{I_{n-1}})}\\
	& = & \prod_{i=1}^{k-1}\frac{1}{1-e^{-\bar{y}_{I_j}}}\\
\end{eqnarray*}
where the indices $I_1,\ldots, I_{n-1}$ are defined such that $\{\bar{e}_{I_1},\ldots, \bar{e}_{i_{n-1}}\}$ is orthogonal to the basis $\{e_{i_1}-e_{j_1},\ldots,e_{i_{n-1}}-e_{j_{n-1}}\}.$

In the functional representation of type $\hat{\mathcal{P}}_1^n$ we have
\begin{eqnarray*}
	\lbrack \langle e_{i_1}-e_{j_1},\ldots,e_{i_{n-1}}-e_{j_{n-1}} \rbrack & \mapsto & \mathcal{L}_d\left(\lbrack \langle e_{i_1}-e_{j_1},\ldots,e_{i_{n-1}}-e_{j_{n-1}}\rangle\right) \\
	& = & \sum_{k_1,\ldots, k_{n-1}=0}^\infty (x_{i_1}/x_{j_i})^{k_1}\cdots(x_{i_{n-1}}/x_{j_{n-1}})^{k_{n-1}}\\
	& = & \prod_{k=1}^{n-1} \frac{1}{1-x_{i_k}/x_{j_k}}.
\end{eqnarray*}
In the functional representation of type $\mathcal{P}_1^n$ we have
$$\lbrack \langle e_{i_1}-e_{j_1},\ldots,e_{i_{n-1}}-e_{j_{n-1}} \rbrack\mapsto \mathcal{L}_\mathcal{I}\left(\lbrack \langle e_{i_1}-e_{j_1},\ldots,e_{i_{n-1}}-e_{j_{n-1}}\rangle\right)
=\int_{t_1,\ldots, t_{n-1}=0}^\infty e^{-\sum t_k(y_{i_k}-y_{j_k})}dt$$
$$=\prod_{k=1}^{n-1} \frac{1}{y_{i_k}-y_{j_k}}.$$

%
%
\end{proof}

It may be interesting to compare the simplicial cone in Example \ref{example: simplex in root solid triangulation} with the simplices appearing in the triangulation of the type $A_3$ root polytope, in \cite{PostnikovPermutohedra}.  The cone below is also represented functionally (using functions which are of type $\mathcal{P}_q^4$) in \cite{HeSon}.  There the functions are called \textit{Cayley functions}.  The computation is included below.
\begin{example}\label{example: simplex in root solid triangulation}
	Let $\mathcal{T} = \{(1,3),(2,3),(1,4)\}$ be a tree.  Then we have
	$$\lbrack\langle e_1-e_3,e_2-e_3,e_1-e_4\rangle\rbrack = \lbrack\lbrack 1,2,3,4\rbrack\rbrack + \lbrack\lbrack 1,2,4,3\rbrack\rbrack + \lbrack\lbrack 1,4,2,3\rbrack\rbrack + \lbrack\lbrack 2,1,3,4\rbrack\rbrack + \lbrack\lbrack 2,1,4,3\rbrack\rbrack$$
	$$-\left(\lbrack\lbrack 12,3,4\rbrack\rbrack +\lbrack\lbrack 12,4,3\rbrack\rbrack +\lbrack\lbrack 1,24,3\rbrack\rbrack +\lbrack\lbrack 1,2,34\rbrack\rbrack  + \lbrack\lbrack 2,1,34\rbrack\rbrack \right)$$
	$$+\lbrack\lbrack 12,34\rbrack\rbrack.$$
In the functional representation of type $\mathcal{P}^4$ we have
\begin{small}
	\begin{eqnarray*}
&& 	\frac{1}{\left(1-x_1\right) \left(1-x_1 x_2\right) \left(1-x_1 x_2 x_3\right)} + \frac{1}{\left(1-x_1\right) \left(1-x_1 x_2\right) \left(1-x_1 x_2 x_4\right)} + \frac{1}{\left(1-x_1\right) \left(1-x_1 x_4\right) \left(1-x_1 x_2 x_4\right)} \\
	&  + & \frac{1}{\left(1-x_2\right) \left(1-x_1 x_2\right) \left(1-x_1 x_2 x_3\right)} + \frac{1}{\left(1-x_2\right) \left(1-x_1 x_2\right) \left(1-x_1 x_2 x_4\right)}\\
	&  - & \frac{1}{\left(1-x_1 x_2\right) \left(1-x_1 x_2 x_3\right)}-\frac{1}{\left(1-x_1 x_2\right) \left(1-x_1 x_2 x_4\right)} - \frac{1}{\left(1-x_1\right) \left(1-x_1 x_2 x_4\right)} - \frac{1}{\left(1-x_1\right) \left(1-x_1 x_2\right)} \\
	& - & \frac{1}{\left(1-x_2\right) \left(1-x_1 x_2\right)}\\
	& + & \frac{1}{1-x_1 x_2}\\
	& = & \frac{1-x_1^2 x_2 x_3 x_4}{\left(1-x_1\right) \left(1-x_2\right) \left(1-x_1 x_2 x_3\right) \left(1-x_1 x_4\right)}\\
	& = & \frac{1}{ \left(1-x_2\right) \left(1-x_1 x_2 x_3\right) \left(1-x_1 x_4\right)},
\end{eqnarray*}
since $x_1x_2x_3x_4=1$.
\end{small}

Let us also illustrate the functional representation of type $\mathcal{P}_1^4$.  In that case we have
\begin{small}
	\begin{eqnarray*}
 & & \frac{1}{\left(x_1-x_2\right) \left(x_2-x_3\right) \left(x_3-x_4\right)} + \frac{1}{\left(x_1-x_2\right) \left(x_2-x_4\right) \left(x_4-x_3\right)} + \frac{1}{\left(x_1-x_4\right) \left(x_4-x_2\right) \left(x_2-x_3\right)}\\
& + &  \frac{1}{\left(x_2-x_1\right) \left(x_3-x_4\right) \left(x_4-x_2\right)} +\frac{1}{\left(x_2-x_1\right) \left(x_1-x_4\right) \left(x_4-x_3\right)}\\
& = & \frac{1}{\left(x_1-x_3\right) \left(x_2-x_3\right) \left(x_1-x_4\right)}.
	\end{eqnarray*}
\end{small}
\end{example}

\begin{figure}[h!]
	\centering
	\includegraphics[width=0.5\linewidth]{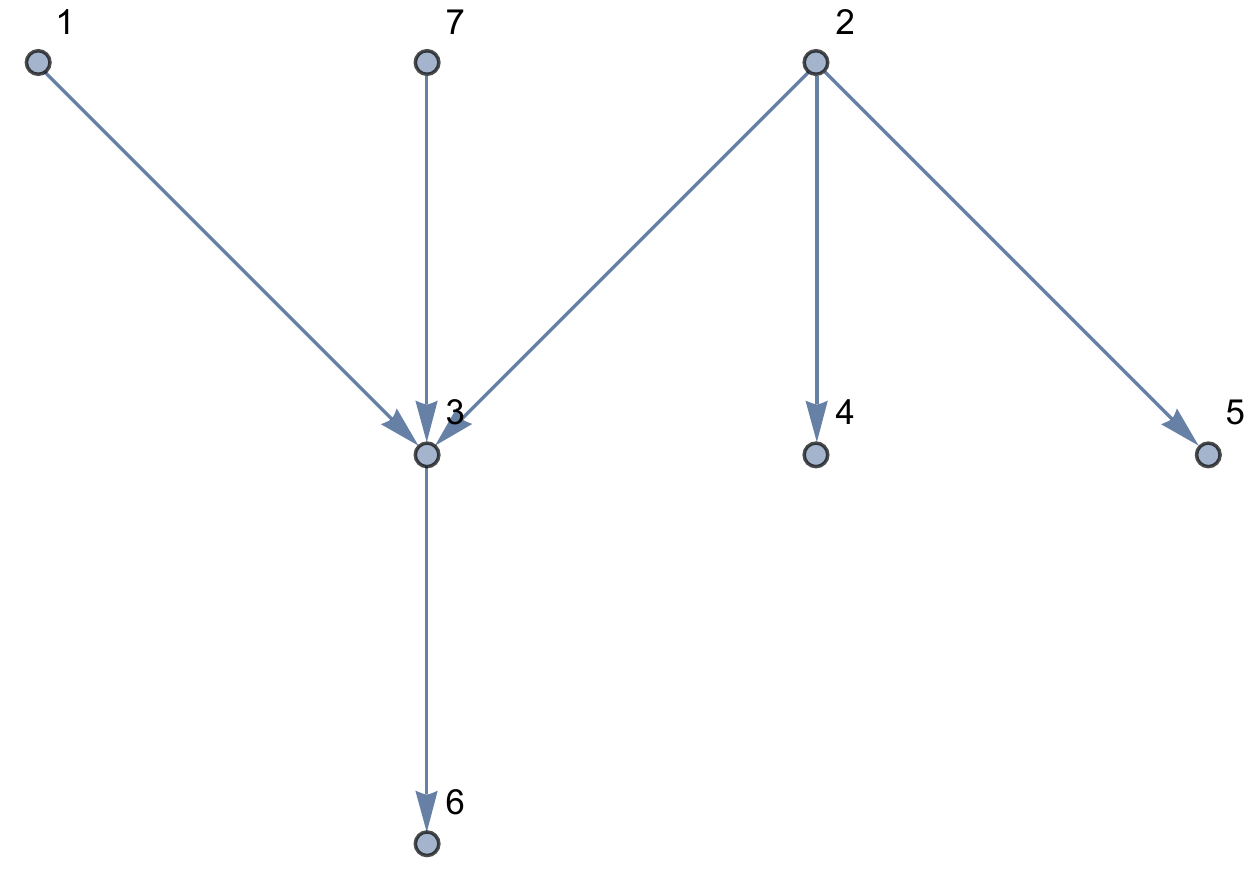}
	\caption{The tree which encodes the orientations in Example \ref{example:7-particle computation}}
	\label{fig:feynman-diagram-polytope-example-tree}
\end{figure}

We present the results of a more involved example which can be worked with the aid of a computer.

\begin{example}\label{example:7-particle computation}
	The directed edges for the tree in Figure \ref{fig:feynman-diagram-polytope-example-tree} are labeled by the ordered pairs 
	$$(1,3),(2,3),(2,4),(2,5),(3,6),(7,3)$$
	and encode the permutohedral cone
	$$\langle e_1-e_3,e_2-e_3,e_2-e_4,e_2-e_5,e_3-e_6,e_7-e_3\rangle_+.$$
	We apply the combinatorial expansion from Theorem \ref{thm:treeExpansionCompletePlates} to give functional representations of this cone of types $\mathcal{P}^n,\hat{\mathcal{P}}_1^n,\mathcal{P}_1^n$, ignoring respectively higher codimensional cones, non-pointed cones and both.  In the functional representation of type $\mathcal{P}_1^n$ we have a sum of fractions over 124 ordered-pair compatible permutations, which eventually collapses to 
	$$\frac{1}{\left(x_1-x_3\right) \left(x_2-x_3\right) \left(x_2-x_4\right) \left(x_2-x_5\right) \left(x_3-x_6\right) \left(x_7-x_3\right)}.$$	
	In the functional representation of type $\hat{\mathcal{P}}_1^n$ we find
	$$\frac{x_3^3 x_4 x_5 x_6}{\left(x_1-x_3\right) \left(x_2-x_3\right) \left(x_2-x_4\right) \left(x_2-x_5\right) \left(x_3-x_6\right) \left(x_7-x_3\right)}$$
	$$=\frac{1}{(1-x_1/x_3)(1-x_2/x_3)(1-x_2/x_4)(1-x_2/x_5)(1-x_3/x_6)(1-x_7/x_3)}.$$
	Recall that there is additional structure that is being missed in $\mathcal{P}_1^n$, and that it can be traced back to variations of a single linear relation.  For $\mathcal{P}_1^n$ we have from the integral Laplace transform
	$$\lbrack\lbrack i,j\rbrack+\lbrack\lbrack j,i\rbrack\rbrack \mapsto \frac{1}{x_i-x_j}+\frac{1}{x_j-x_i}=0,$$
	while for $\hat{\mathcal{P}}_1^n$ we have from the discrete Laplace transform
	$$\lbrack\lbrack i,j\rbrack\rbrack+\lbrack\lbrack j,i\rbrack\rbrack \mapsto \frac{1}{1-x_i/x_j}+\frac{1}{1-x_j/x_i}=1,$$
	where $1$ corresponds to a point, the origin, under the discrete Laplace transform.
	
	Note that in $\hat{\mathcal{P}}^n$ we have an additional lumped term $\lbrack ij\rbrack$, which is the whole line 
	$$\lbrack ij\rbrack = \{x\in V_0^n:x_{i}+x_j=0\text{ and } x_a=0 \text{ for all } a\not\in\{i,j\}\},$$
	hence 
	$$\lbrack\lbrack i,j\rbrack\rbrack+\lbrack\lbrack j,i\rbrack\rbrack=\lbrack\lbrack i,j\rbrack\rbrack\cdot \lbrack\lbrack j,i\rbrack\rbrack+\lbrack\lbrack ij\rbrack\rbrack.$$
	In $\mathcal{P}^n$ the characteristic function  $\lbrack\lbrack i,j\rbrack\rbrack\cdot \lbrack\lbrack j,i\rbrack\rbrack$ of the point $(0,\ldots, 0)$ is mapped to zero, while for $\hat{\mathcal{P}}_1^n$ the lumped term $\lbrack\lbrack ij\rbrack\rbrack$ is mapped to zero.  In $\mathcal{P}_1^n$ both are mapped to zero.  
	
	Starting in $\hat{\mathcal{P}}^n$, applying Theorem \ref{thm:treeExpansionCompletePlates} we have an alternating sum of 653 fractions labeled by ordered set partitions.  After implementing on the computer, repeatedly imposing the condition $x_1\cdots x_7=1$ it collapses to
	$$\frac{x_1 x_2 x_3 x_4^2 x_5^2 x_6^2}{\left(1-x_1\right) \left(1-x_4\right) \left(1-x_5\right) \left(1-x_2 x_4 x_5\right) \left(1-x_6\right) \left(1-x_1 x_2 x_3 x_4 x_5 x_6\right)}$$
	$$=\frac{1}{\left(1-x_1\right) \left(1-x_7\right)\left(1-x_2 x_4 x_5\right) \left(1-x_1 x_2 x_3 x_4 x_5 x_7\right) \left(1-x_1 x_2 x_3 x_5 x_6 x_7\right) \left(1-x_1 x_2 x_3 x_4x_6x_7\right) }.$$
	
	After the change of variable $x_i=e^{-y_i}$, the domain of convergence coincides with the original permutohedral cone
	$$\langle e_1-e_3,e_2-e_3,e_2-e_4,e_2-e_5,e_3-e_6,e_7-e_3\rangle_+$$ 
	and is cut out by the system of inequalities
	\begin{eqnarray*}
		\left\{y\in\mathbb{R}^7:\begin{array}{cccc}
			y_1 \ge 0	& y_7 \ge 0 & y_{245}\ge 0 &  \\ 
			y_{123457} \ge 0 & y_{123567} \ge 0 & y_{123467} \ge 0 &  \\ 
			& y_{1234567} =0& 
		\end{array} \right\},
	\end{eqnarray*}
	which are equivalently determined by requiring nonnegativity of the dot product $y\cdot \bar{e}_I$ (where $y\in V_0^n$) for each $\bar{e}_I$ in the basis 
	$$\left\{\bar{e}_{1},\bar{e}_{7},\bar{e}_{245},\bar{e}_{123457},\bar{e}_{123567},\bar{e}_{123467}\right\}$$
	of $V_0^n$ which is orthogonal to the root basis determined by the ordered pair data as encoded in Figure \ref{fig:feynman-diagram-polytope-example-tree},
	$$e_1-e_3,e_7-e_3,e_2-e_3,e_3-e_6,e_2-e_4,e_2-e_5,$$
	where as usual $\bar{e}_I = \frac{\vert I^c\vert }{n}e_I - \frac{\vert I\vert}{n} e_{I^c}$ for $I$ a proper subset of $\{1,\ldots, n\}$.
\end{example}


\end{document}